\documentclass[reqno,11pt]{amsart}        
\usepackage{amsfonts,amsmath,amssymb,amsthm,colordvi,mathrsfs,graphicx}
\usepackage{caption,subcaption,float}
\usepackage[utf8]{inputenc}
\usepackage{mathrsfs}
\usepackage{geometry}
\usepackage{enumerate}

\usepackage{amsmath, amssymb, amsthm, geometry, hyperref}

\usepackage{tabularx}
\usepackage{tabulary}
 
\hypersetup{
  colorlinks=true,
  linkcolor=blue,
  citecolor=blue,
  urlcolor=blue
}
\usepackage[all,knot,poly]{xy}
\usepackage{tikz-cd}
\usepackage{tikz}
\usepackage{verbatim}
\usetikzlibrary{arrows}
\usetikzlibrary{knots}
\tikzset{every path/.style={line width=0.4pt},every node/.style={transform shape,knot crossing,inner sep=1.5pt},>=triangle 60,text node/.style={rectangle,transform shape=false,black}}
 \usetikzlibrary{decorations.markings, arrows.meta}

\theoremstyle{plain}

\newcommand{\R}{\mathbb R}
\newcommand{\C}{\mathbb C}
\newcommand{\PP}{\mathbb P}

\newcommand{\Z}{\mathbb Z}

\newcommand{\Log}{\mathrm{Log}\,}

\newcommand{\di}{\displaystyle}

\newcommand\restr[2]{{% we make the whole thing an ordinary symbol
  \left.\kern-\nulldelimiterspace % automatically resize the bar with \right
  #1 % the function
  \right|_{#2} % this is the delimiter
  }}

\newtheorem{principle}{Principle}

\newtheorem{example}{Example}

\newtheorem{remark}{Remark}[section]
\newtheorem*{mtheorem*}{Main Theorem}
\newtheorem{theorem}{Theorem}[section]
\newtheorem{definition}{Definition}[section]
\newtheorem{proposition}{Proposition}[section]
\newtheorem{corollary}{Corollary}[section]
\newtheorem{lemma}{Lemma}[section]
\begin{document}
\title{From Logarithmic Limit Sets to Algebraicity}%\title{TBA}%[A tropical characterization of complex analytic varieties]{A tropical characterization of complex analytic varieties to be algebraic}
\author{Mounir Nisse}
\date{}

\newcommand{\hooklongrightarrow}{\lhook\joinrel\longrightarrow}
\newcommand{\twoheadlongrightarrow}{\relbar\joinrel\twoheadrightarrow}

\vbadness=10000
\hbadness=10000
\tolerance=2000

%%%%%%%%%%%%%%%%%%%%%%%%%%%%%%%%%%%%%%%%%%%%%%%%%%%%%%%%%%%%%%%%%%%%%%%%%%%%%% 
 

\address{Mounir Nisse\\
Department of Mathematics, Xiamen University Malaysia, Jalan Sunsuria, Bandar Sunsuria, 43900, Sepang, Selangor, Malaysia.
}
\email{mounir.nisse@gmail.com, mounir.nisse@xmu.edu.my}
\thanks{}
\thanks{This research is supported in part by Xiamen University Malaysia Research Fund (Grant no. XMUMRF/ 2020-C5/IMAT/0013).}

\maketitle

%%%%%%%%%%%%%%%%%%%%%%%%%%%%%%%%%%%%%%%%%%

\begin{abstract}
We study a converse problem to the theorem of Bergman and Bieri--Groves on logarithmic limit sets of algebraic subvarieties of the complex algebraic torus $(\mathbb C^*)^n$. We introduce the notion of \emph{finite logarithmic type}, a boundary condition formulated on toric compactifications in terms of coherent meromorphic extensions with uniformly bounded pole orders along toric boundary divisors. We prove that a closed reduced analytic subvariety of $(\mathbb C^*)^n$ whose logarithmic limit set is a finite rational spherical polyhedral complex of the expected dimension is algebraic whenever it is of finite logarithmic type. The proof relies on toric compactifications, coherent extension theory, Serre's GAGA theorem, and Chow's theorem.
A principal result of the paper is the complete treatment of the one-dimensional case. We prove that every closed analytic curve in $(\mathbb C^*)^n$ with finite logarithmic limit set is algebraic. Consequently, the finiteness of the logarithmic limit set completely characterizes algebraicity for analytic curves, yielding a genuine converse to Bergman's theorem in dimension one. These results establish new links between logarithmic limit sets, tropical geometry, toric geometry, and complex analytic geometry.
\end{abstract}

%%%%%%%%%%%%%%%%%%%%%%%%%%%%%%%%%%%%%%%%%%%%%%%%%%%%%%%%%%%%%%%%%%%%%%%%%%%%

\section*{Introduction}

One of the classical themes in algebraic geometry is the problem of recognizing when an analytic subvariety of an algebraic ambient space is itself algebraic. The prototype of this philosophy is Chow's theorem \cite{Chow49}, which asserts that every closed analytic subvariety of complex projective space is algebraic. In the noncompact setting the situation is substantially more delicate. The algebraic torus $T=(\mathbb C^*)^n$ admits many closed analytic subvarieties that are not algebraic, and therefore some additional condition controlling the behavior at infinity is necessary.

The asymptotic geometry of analytic subvarieties of the torus is naturally encoded by amoebas and logarithmic limit sets. Since the pioneering work of Bergman \cite{Bergman71} and the subsequent development by Bieri and Groves \cite{BieriGroves84}, it has become clear that logarithmic limit sets provide a tropical shadow of algebraic varieties. If $V\subset (\mathbb C^*)^n$ is algebraic, then the logarithmic limit set $\mathscr L^\infty(V)$ is a finite rational spherical polyhedral complex of dimension $\dim_{\mathbb C}(V)-1$. This fundamental result establishes a remarkable bridge between algebraic geometry and polyhedral geometry and lies at the foundation of tropical geometry; see also \cite{EKL06,MaclaganSturmfels15}.

A natural converse question arises. Suppose that $V\subset (\mathbb C^*)^n$ is a closed analytic subvariety whose logarithmic limit set is finite and rational polyhedral. Must $V$ be algebraic? This question may be viewed as a converse problem to the Bergman--Bieri--Groves theorem \cite{Bergman71,BieriGroves84}. Although the rational polyhedrality of $\mathscr L^\infty(V)$ imposes strong restrictions on the asymptotic directions of the amoeba, it does not by itself control the analytic complexity of the variety near infinity. Essential singularities may still occur along boundary directions and prevent algebraicity.

The purpose of the present paper is to identify a precise boundary finiteness condition that eliminates this obstruction. The central notion introduced here is the concept of finite logarithmic type. Roughly speaking, finite logarithmic type means that after choosing a suitable toric compactification $X_\Sigma$ adapted to the logarithmic limit set, the analytic equations defining the variety extend meromorphically across the toric boundary with uniformly bounded pole orders. Equivalently, the ideal sheaf of the variety admits a coherent meromorphic extension whose poles are controlled by one fixed torus-invariant divisor. This perspective is closely related to the philosophy of toric compactifications and tropical compactifications developed by Tevelev \cite{Tevelev07}.

The main theorem of the paper shows that finite logarithmic type is sufficient for algebraicity. More precisely, if $V\subset (\mathbb C^*)^n$ is a closed reduced analytic subvariety with finitely many irreducible components, if $\mathscr L^\infty(V)$ is a finite rational spherical polyhedral complex of dimension $\dim_{\mathbb C}(V)-1$, and if $V$ is of finite logarithmic type with respect to an adapted complete toric compactification, then $V$ is algebraic. In fact, $V$ is cut out by finitely many Laurent polynomials.

\begin{theorem}[Finite Logarithmic Type Theorem]
Let $V\subset (\mathbb C^*)^n$ be a closed reduced analytic subvariety with finitely many irreducible components. Assume that $\mathscr L^\infty(V)$ is a finite rational spherical polyhedral complex of dimension $\dim_{\mathbb C}(V)-1$. Let $\Sigma$ be a complete rational fan refining the conical decomposition determined by $\mathscr L^\infty(V)$. If $V$ is of finite logarithmic type with respect to the toric compactification $X_\Sigma$, then $V$ is algebraic. More precisely, $V$ is defined by finitely many Laurent polynomials.
\end{theorem}

A major motivation for introducing finite logarithmic type comes from the one-dimensional case. Curves provide the simplest setting in which the interaction between logarithmic limit sets and algebraicity can be analyzed completely. One of the principal results of this paper establishes a new algebraicity theorem for analytic curves.

%%%%%%%%%%%%%%%%%%%%%%%%%%%%%%%%
\medskip

A primary feature of the paper is the complete treatment of the one-dimensional case. Let $C\subset(\mathbb C^\ast)^n$ be a closed analytic curve with finitely many irreducible components. In this case the logarithmic limit set has expected dimension $0$. We prove that if $\mathscr L^\infty(C)$ is finite, then $C$ is algebraic. In particular, if $\mathscr L^\infty(C)$ is a finite rational set, then $C$ is algebraic and hence is automatically of finite logarithmic type. This gives a genuine converse to Bergman's theorem in dimension one.  It shows that for analytic curves the finiteness and rationality of the logarithmic limit set completely characterize algebraicity. The result highlights the remarkable rigidity of one-dimensional analytic geometry inside the algebraic torus.

\begin{theorem}[One-Dimensional Algebraicity Theorem] 
Let $C\subset (\mathbb C^*)^n$ be a closed analytic curve with finitely many irreducible components. If the logarithmic limit set $\mathscr L^\infty(C)$ is finite, then $C$ is algebraic.
\end{theorem} 

\medskip

The proof in dimension one is conceptually different from the higher-dimensional argument. Since $\mathscr L^\infty(C)$ is finite, one may choose a torus character $\chi^\nu$ such that $\langle\nu,\xi\rangle\neq0$ for every logarithmic limiting direction $\xi$. The restriction $f=\chi^\nu|_C:C\to\mathbb C^\ast$ is then proper. After normalization, this gives a finite holomorphic map from the normalized curve to $\mathbb C^\ast$, and hence the normalization compactifies by adding finitely many points over $0$ and $\infty$. Near an added point, the character has the form $s^q u(s)$ with $q\neq0$ and $u(0)\neq0$. The key analytic lemma shows that if any coordinate function had an essential singularity at such a puncture, then the normalized logarithmic vectors would have infinitely many limiting directions. This contradicts the finiteness of $\mathscr L^\infty(C)$. Therefore all coordinate functions extend meromorphically to the compactified normalization. The resulting map to $(\mathbb P^1)^n$ has compact analytic image, and Chow's theorem implies that the image is algebraic. Intersecting with the dense torus gives the algebraicity of $C$.

%%%%%%%%%%%%%%%%%%%%%%%%%%%%%%%%
\medskip

This theorem may be regarded as a genuine converse to the one-dimensional Bergman--Bieri--Groves theorem \cite{Bergman71,BieriGroves84}. It shows that for analytic curves the finiteness and rationality of the logarithmic limit set completely characterize algebraicity. The result highlights the remarkable rigidity of one-dimensional analytic geometry inside the algebraic torus.

\medskip

A second contribution of the paper is the development of several equivalent formulations of finite logarithmic type. Besides the original definition in terms of coherent extensions of ideal sheaves, we establish an intrinsic valuative characterization. In this form, finite logarithmic type is expressed by the existence of a coherent meromorphic extension $\mathcal F\subset \mathcal M_{X_\Sigma}$ satisfying uniform inequalities $\operatorname{ord}_{D_\rho}(f)\ge -M_\rho$ along every toric boundary divisor $D_\rho$. This formulation makes clear that finite logarithmic type is fundamentally a divisorial condition measuring the growth of analytic equations at infinity.

The paper also clarifies the relationship between finite logarithmic type and GAGA. In the classical projective setting, coherence together with properness leads to algebraicity through Serre's GAGA theorem \cite{SerreGAGA56}. In the toric setting, finite logarithmic type plays the role of a logarithmic properness condition. The bounded-pole condition allows coherent analytic data on the torus to extend to coherent analytic sheaves on a proper toric compactification, after which GAGA and Chow's theorem \cite{Chow49} become applicable. This interpretation places finite logarithmic type within a broader framework connecting algebraic geometry, analytic geometry, and compactification theory.

Another contribution is the study of analytic closures in toric compactifications. We show that if the closure of an analytic subvariety in a toric compactification is already analytic, then finite logarithmic type automatically holds. This provides a geometric source of examples and explains the naturality of the bounded-pole condition. From this perspective, finite logarithmic type measures the extent to which an analytic variety behaves as though it had an analytic closure along the toric boundary.

The results obtained here suggest a broader program. The central open problem is whether finite logarithmic type follows automatically from the rational polyhedrality of the logarithmic limit set. A positive answer would transform the main theorem into a complete converse to the Bergman--Bieri--Groves theorem and would provide a purely tropical criterion for algebraicity in the complex torus.

\medskip

The philosophy emerging from this work may be summarized as follows. The logarithmic limit set records the combinatorial skeleton of an analytic variety at infinity, toric compactifications convert this combinatorial information into boundary geometry \cite{Fulton93}, finite logarithmic type controls the analytic behavior along the boundary, and GAGA together with Chow's theorem transforms this information into algebraicity. The one-dimensional theorem shows that in the case of curves no additional hypothesis is needed. Consequently the results of this paper by the introduction of finite logarithmic type 
provide new evidence that tropical asymptotic geometry contains far more algebraic information than previously recognized and open several new directions for future investigation.

%%%%%%%%%%%%%%%%%%%%%%%%%%%%%%%%%%%%%%%%%%%%%%%%%%%%%%%%%%%%%%%%%%%%%%%%%%%%

%%%%%%%%%%%%%%%%%%%%%%%%%%%%%%%%%%%%%%%%%%%%%%%%%%%%%%%%%%%%%%%%%%%%%%%%%%%%%%

The principal results and contributions of this paper may be summarized as follows.
First, we introduce the notion of finite logarithmic type for analytic subvarieties of the algebraic torus and show that it provides the appropriate boundary finiteness condition for deriving algebraicity from tropical asymptotic data.
Second, we prove an algebraicity theorem stating that a closed analytic subvariety with finite rational logarithmic limit set and finite logarithmic type is algebraic.
Third, we establish a complete converse to the Bergman theorem in dimension one by proving that every closed analytic curve in $(\mathbb C^*)^n$ with finite logarithmic limit set is algebraic.
Fourth, we introduce a scalar logarithmic growth lemma and a local monomial meromorphic-extension theorem, both of which appear to be new and may be of independent interest in the study of analytic varieties in algebraic tori.
Finally, we develop several equivalent formulations of finite logarithmic type and clarify its relation with toric compactifications, coherent meromorphic extensions, and tropical compactification theory.

%%%%%%%%%%%%%%%%%%%%%%%%%%%%%%%%%%%%%%%%%%%%%%%%%%%%%%%%%%%%%%%%%%%%%%%%%%%%%%

\subsection*{Organization of the paper}

Section~1 recalls the necessary background on logarithmic limit sets, toric varieties, coherent analytic sheaves, and tropical compactifications. In Section~2 we study the one-dimensional case and prove that a closed analytic curve in $(\mathbb{C}^*)^n$ with finite logarithmic limit set is algebraic. Section~3 introduces the notion of finite logarithmic type and develops its fundamental properties together with several equivalent formulations. In Section~4 we prove the main algebraicity theorem, showing that finite logarithmic type together with a finite rational logarithmic limit set implies algebraicity. Section~5 discusses the relation with tropical compactifications and the work of Tevelev. The final section contains examples, and consequences.

%%%%%%%%%%%%%%%%%%%%%%%%%%%%%%%%%%%%%%%%%%%%%%%%%%%%%%%%%%%%%%%%%%%%%%%%%%%%%%

\section{Preliminaries}

Throughout this paper we work over the field $\mathbb C$. We denote by $T=(\mathbb C^*)^n$ the complex algebraic torus of dimension $n$. The character lattice of $T$ is $M=\mathrm{Hom}(T,\mathbb C^*)\simeq \mathbb Z^n$ and the lattice of one-parameter subgroups is $N=\mathrm{Hom}(M,\mathbb Z)$. We write $M_{\mathbb R}=M\otimes_{\mathbb Z}\mathbb R$ and $N_{\mathbb R}=N\otimes_{\mathbb Z}\mathbb R$. The natural perfect pairing is denoted by $\langle \cdot,\cdot\rangle:M\times N\to\mathbb Z$. For $m\in M$ we write $\chi^m$ for the corresponding character of the torus.

\medskip

For a point $z=(z_1,\ldots,z_n)\in T$, the logarithmic map is defined by $\Log(z)=(\log|z_1|,\ldots,\log|z_n|)\in\mathbb R^n$. If $V\subset T$ is an analytic subvariety, its amoeba is the subset $\mathcal A(V)=\Log(V)\subset \mathbb R^n$. The amoeba is closed whenever $V$ is closed analytic. The asymptotic geometry of $V$ is encoded by the behavior of $\mathcal A(V)$ at infinity.

\medskip

Let $X$ be a connected, locally compact, noncompact Hausdorff space. An end of $X$ in the sense of Freudenthal is an element of the inverse limit $\varprojlim_K \pi_0(X\setminus K)$, where $K$ ranges over compact subsets of $X$. Equivalently, an end is represented by a nested sequence of connected components escaping every compact subset. The set of ends measures the topology of $X$ at infinity.
The logarithmic limit set of a subset $V\subset T$ is defined by
$$
\mathscr L^\infty(V)=\left\{\lim_{j\to\infty}\frac{x_j}{\|x_j\|}\; ;\; x_j\in \Log(V),\ \|x_j\|\to\infty\right\}\subset S^{n-1}.
$$
This set records the asymptotic directions of the amoeba. If $V$ is algebraic, the theorem of Bergman, Bieri and Groves asserts that $\mathscr L^\infty(V)$ is a finite rational spherical polyhedral complex of dimension $\dim_{\mathbb C}(V)-1$. This result plays a central role in tropical geometry and provides the link between algebraic varieties and polyhedral structures.
A rational polyhedral cone in $N_{\mathbb R}$ is a subset of the form $\sigma=\mathbb R_{\ge 0}u_1+\cdots+\mathbb R_{\ge 0}u_r$, where $u_1,\ldots,u_r\in N$. A face of $\sigma$ is the intersection of $\sigma$ with a supporting hyperplane. A cone is strongly convex if $\sigma\cap(-\sigma)=\{0\}$. A fan $\Sigma$ in $N_{\mathbb R}$ is a finite collection of strongly convex rational polyhedral cones satisfying the conditions that every face of a cone of $\Sigma$ belongs to $\Sigma$ and the intersection of any two cones of $\Sigma$ is a common face.
To a fan $\Sigma$ one associates a toric variety $X_\Sigma$. If $\sigma\in\Sigma$, the corresponding affine toric variety is
$$
U_\sigma=\mathrm{Spec}\,\mathbb C[\sigma^\vee\cap M],
$$
where $\sigma^\vee=\{m\in M_{\mathbb R}\,;\,\langle m,u\rangle\ge 0 \ \text{for all}\ u\in\sigma\}$ is the dual cone. The affine varieties $U_\sigma$ glue along common open subsets and produce the toric variety $X_\Sigma$. The torus $T$ is a dense open subset of $X_\Sigma$. The toric variety is complete if and only if the support $|\Sigma|$ equals $N_{\mathbb R}$.

\medskip

Every cone $\sigma\in\Sigma$ corresponds to a torus orbit $O(\sigma)\subset X_\Sigma$. The orbit-cone correspondence gives a bijection between cones of $\Sigma$ and torus orbits. The closure of the orbit $O(\sigma)$ is itself a toric variety associated with the star of $\sigma$. If $\dim\sigma=k$, then $\dim O(\sigma)=n-k$. In particular, the rays $\rho\in\Sigma(1)$ correspond to irreducible torus-invariant divisors.
For a ray $\rho\in\Sigma(1)$, let $u_\rho$ denote its primitive lattice generator and let $D_\rho$ be the corresponding invariant prime divisor. Every torus-invariant Weil divisor has the form $D=\sum_{\rho\in\Sigma(1)}a_\rho D_\rho$. If $m\in M$, the divisor of the character $\chi^m$ is
$$
\mathrm{div}(\chi^m)=\sum_{\rho\in\Sigma(1)}\langle m,u_\rho\rangle D_\rho.
$$
A torus-invariant Weil divisor $D=\sum a_\rho D_\rho$ is Cartier if and only if for every cone $\sigma\in\Sigma$ there exists an element $m_\sigma\in M$ satisfying $\langle m_\sigma,u_\rho\rangle=-a_\rho$ for every ray $\rho$ contained in $\sigma$.

\medskip

Support functions provide a convenient description of Cartier divisors. A support function on $\Sigma$ is a continuous function $\psi:|\Sigma|\to\mathbb R$ such that the restriction of $\psi$ to each cone $\sigma$ is integral linear. If $\psi|_\sigma=\langle m_\sigma,\cdot\rangle$, then the associated Cartier divisor is $D_\psi=-\sum_\rho \psi(u_\rho)D_\rho$. The correspondence between integral support functions and invariant Cartier divisors is bijective.
Let $\mathcal O_{X_\Sigma}(D)$ be the line bundle associated with a Cartier divisor $D$. The global sections of $\mathcal O_{X_\Sigma}(D)$ admit a combinatorial description in terms of lattice points of the polyhedron associated with $D$. This relationship between divisors and polyhedra is one of the basic features of toric geometry.

\medskip

We now recall some analytic notions. If $X$ is a complex analytic space, a sheaf $\mathcal F$ of $\mathcal O_X$-modules is coherent if every point of $X$ has a neighborhood on which $\mathcal F$ admits a finite presentation. By Oka's coherence theorem, the sheaf of holomorphic functions is coherent. By Cartan's theorem, the ideal sheaf of a closed analytic subspace is coherent.
The sheaf of meromorphic functions on $X$ is denoted by $\mathcal M_X$. A coherent meromorphic sheaf is a coherent $\mathcal O_X$-submodule of $\mathcal M_X^r$ for some integer $r$. If $\mathcal F\subset \mathcal M_X$ is coherent, its local sections are meromorphic functions whose poles are controlled by finitely many divisorial conditions.

\medskip

Let $X_\Sigma$ be a smooth complete toric variety. For an effective invariant divisor $E=\sum_{\rho\in\Sigma(1)}M_\rho D_\rho$, the sheaf $\mathcal O_{X_\Sigma}(E)$ can be viewed as a subsheaf of $\mathcal M_{X_\Sigma}$ and is characterized by the condition
$$
\mathcal O_{X_\Sigma}(E)(U)=\{f\in \mathcal M_{X_\Sigma}(U)\,;\,\mathrm{ord}_{D_\rho}(f)\ge -M_\rho \ \text{for all}\ \rho\}.
$$
Thus sections of $\mathcal O_{X_\Sigma}(E)$ are meromorphic functions whose poles along $D_\rho$ are bounded by the prescribed integers $M_\rho$.
If $V\subset T$ is a closed reduced analytic subvariety, we denote by $\mathcal I_V\subset\mathcal O_T$ its ideal sheaf. A coherent meromorphic extension of $\mathcal I_V$ to $X_\Sigma$ is a coherent analytic subsheaf $\mathcal F\subset\mathcal M_{X_\Sigma}$ such that $\mathcal F|_T=\mathcal I_V$. If there exists an effective divisor $E$ with $\mathcal F\subset\mathcal O_{X_\Sigma}(E)$, then the poles of all local sections of $\mathcal F$ are uniformly bounded by $E$.

\medskip

Several fundamental analytic theorems will be used. The Remmert proper mapping theorem states that the image of an analytic set under a proper holomorphic map is analytic. Bishop's compactness theorem \cite{Bishop64} asserts that a family of analytic cycles with uniformly bounded volume in a compact region admits convergent subsequences. These results provide compactness and extension properties for analytic varieties.
A holomorphic map $f:X\to Y$ between analytic spaces is finite if it is proper and has finite fibers. By the Weierstrass preparation theorem, finite holomorphic maps are locally described by finite modules over rings of holomorphic functions. This property is repeatedly used when studying projections of analytic varieties.

\medskip

We now recall the algebraic-analytic comparison theorems. If $X$ is a proper algebraic variety over $\mathbb C$, Serre's GAGA theorem asserts that analytification induces an equivalence between coherent algebraic sheaves on $X$ and coherent analytic sheaves on $X^{an}$. Moreover, the cohomology groups of coherent algebraic sheaves agree with those of their analytifications. As a consequence, coherent analytic ideal sheaves on proper algebraic varieties are algebraizable.

\medskip

Chow's theorem states that every closed analytic subvariety of complex projective space is algebraic. More generally, if $X$ is projective and $Y\subset X^{an}$ is a closed analytic subset, then $Y$ is algebraic. This theorem forms the bridge between complex analytic geometry and algebraic geometry.

\medskip

The tropicalization of an algebraic subvariety of the torus is the support of a weighted rational polyhedral complex in $N_{\mathbb R}$. The Bergman--Bieri--Groves theorem identifies this tropicalization with the cone over the logarithmic limit set. Consequently, the asymptotic behavior of algebraic varieties is governed by polyhedral geometry.
A tropical compactification of a subvariety $V\subset T$ is a compactification obtained from a toric variety whose fan is sufficiently compatible with the tropicalization of $V$. Tevelev's theorem asserts that if $\Sigma$ is a fan supported on the tropicalization of $V$, then the closure $\overline V$ in $X_\Sigma$ intersects precisely those torus orbits corresponding to cones that meet the tropicalization. This theorem provides a geometric interpretation of tropicalization in terms of toric compactifications.

\medskip

Throughout the paper, complete rational fans refining the conical decomposition determined by $\mathscr L^\infty(V)$ will be used. The associated toric compactifications allow one to translate asymptotic information encoded by logarithmic limit sets into geometric information along toric boundary divisors. Combined with coherent meromorphic extensions, GAGA, Chow's theorem, Bishop compactness, Remmert's theorem, and the theory of tropical compactifications, these tools constitute the foundational framework for the proofs of the main results.

\section{The one-dimensional case}

The case of analytic curves already exhibits the essential geometric mechanism behind the general algebraicity criterion. The central observation is that the logarithmic limit set records the asymptotic directions of the amoeba and therefore controls the behavior of the ends of the curve in the algebraic torus. From the tropical point of view, the logarithmic limit set is the spherical quotient of the Bergman fan introduced by Bergman \cite{Bergman71} and further developed by Bieri and Groves \cite{BieriGroves84}. For algebraic varieties, this object is a finite rational polyhedral complex whose dimension is one less than the dimension of the variety. In the case of a curve, the logarithmic limit set consists of finitely many rational points of the sphere and may be viewed as the collection of tropical directions of the ends of the curve.

The philosophy underlying tropical compactifications, developed systematically by Tevelev \cite{Tevelev07}, is that a finite rational tropicalization should determine a compactification with controlled boundary behavior. For algebraic subvarieties of a torus this principle is well understood: the Bergman fan governs toric compactifications adapted to the asymptotic geometry of the variety. The theorem proved below shows that, in dimension one, the converse phenomenon also holds. Namely, if an analytic curve has a logarithmic limit set possessing the same finiteness and rationality properties as an algebraic curve, then the curve must already be algebraic.

The proof combines tropical asymptotic information with classical complex-analytic methods. The finiteness of the logarithmic limit set allows one to choose a torus character whose logarithmic growth is nondegenerate along every end of the curve. This produces a proper holomorphic map to $\C^*$. After normalization, the curve compactifies to a compact Riemann surface, and the character extends to a meromorphic function on the compactification. The asymptotic information encoded in the logarithmic limit set then implies that every torus coordinate extends meromorphically across the boundary points. Consequently the normalization map extends to a holomorphic map into $(\PP^1)^n$, and Chow's theorem yields algebraicity.

\medskip
 
In conclusion, the argument combines logarithmic limit sets in the sense of Bergman \cite{Bergman71} and Bieri--Groves \cite{BieriGroves84}, the normalization theory of analytic curves \cite{Forster81}, the classical removable singularity theorem \cite{Remmert98}, and Chow's theorem \cite{Chow49}. The key point is that the finiteness of the logarithmic limit set imposes strong restrictions on the asymptotic behavior of the coordinate functions near the ends of the curve.

For a subset $E\subset(\C^*)^n$, its logarithmic limit set $\mathscr L^\infty(E)$ is the set of all $\xi\in S^{n-1}$ for which there exists a sequence $z_m\in E$ such that $\|\Log z_m\|\to+\infty$ and $\Log z_m/\|\Log z_m\|\to\xi$, where $\Log z=(\log|z_1|,\ldots,\log|z_n|)$. In the algebraic case, the logarithmic limit set is a finite rational spherical polyhedral complex by the theorem of Bergman \cite{Bergman71}; see also Bieri--Groves \cite{BieriGroves84} and the exposition in \cite{MaclaganSturmfels15}.

\begin{definition}[Cluster set at the puncture]\label{def:cluster-set}
Let $h:\Delta^*\to\C^*$ be holomorphic and let $T(s)=\log(1/|s|)$. Put $w(s)=\log|h(s)|/T(s)$. The cluster set in $\R$ of $\log|h(s)|/T(s)$ as $s\to0$ is the set
$
\operatorname{Cl}_0(w)=\{a\in\R:\text{ there exists a sequence }s_m\in\Delta^*\text{ such that }s_m\to0\text{ and }w(s_m)\to a\}.
$
Equivalently, $a\in\operatorname{Cl}_0(w)$ if and only if for every $\varepsilon>0$ and every $\rho>0$ there exists $s\in\Delta^*$ such that $0<|s|<\rho$ and $|w(s)-a|<\varepsilon$.
\end{definition}

\begin{lemma}[Scalar logarithmic growth lemma]\label{scalar-lemma}
Let $h:\Delta^*\to\C^*$ be holomorphic, and set $T(s)=\log(1/|s|)$. Assume that the cluster set in $\R$ of $\log |h(s)|/T(s)$ as $s\to0$ is finite and nonempty. Then $h$ is meromorphic at $0$.
\end{lemma}

\begin{proof}
The proof is based on the classical removable singularity theorem \cite[Chapter~7]{Remmert98}. Put $w(s)=\log |h(s)|/T(s)$. Since $h$ is holomorphic and nowhere zero on $\Delta^*$, the function $\log |h(s)|$ is harmonic on $\Delta^*$, and $w$ is continuous on every punctured disc $0<|s|<\rho$. We first prove that $w$ is bounded near $0$. Suppose, by contradiction, that $w$ is not bounded above near $0$. Since the cluster set of $w$ in $\R$ is finite and nonempty, choose a real cluster value $a$. Let $b>a$ be arbitrary. Since $a$ is a cluster value, for every $\rho>0$ and every $\varepsilon>0$ there exists $s_1\in\Delta^*$ such that $0<|s_1|<\rho$ and $|w(s_1)-a|<\varepsilon$. Choose $\varepsilon>0$ with $a+\varepsilon<b$. Then $w(s_1)<b$. Since $w$ is not bounded above near $0$, for the same $\rho$ there exists $s_2\in\Delta^*$ such that $0<|s_2|<\rho$ and $w(s_2)>b$. The punctured disc $\{s:0<|s|<\rho\}$ is path connected. Hence there exists a continuous path $\gamma:[0,1]\to\Delta^*$ such that $\gamma(0)=s_1$, $\gamma(1)=s_2$, and $0<|\gamma(t)|<\rho$ for all $t\in[0,1]$. Since $w\circ\gamma$ is continuous and 
$$(w\circ\gamma)(0)<b<(w\circ\gamma)(1),$$ 
the intermediate value theorem gives $t_0\in(0,1)$ such that $(w\circ\gamma)(t_0)=b$. Thus, with $s_\rho=\gamma(t_0)$, one has $0<|s_\rho|<\rho$ and $w(s_\rho)=b$. Taking $\rho_m=1/m$ and repeating the preceding argument for all large $m$, we obtain a sequence $s_m\to0$ such that $w(s_m)=b$ for every $m$. Therefore $b$ is a cluster value of $w$ at $0$. Since the same conclusion holds for every real number $b>a$, the cluster set of $w$ contains the interval $(a,+\infty)$ and is therefore infinite. This contradicts the hypothesis. Hence $w$ is bounded above near $0$.

The same argument applied to $-w$ shows that $w$ is bounded below near $0$. Indeed, the cluster set of $-w$ is the negative of the finite nonempty cluster set of $w$. If $-w$ were not bounded above near $0$, then the preceding argument would imply that the cluster set of $-w$ is infinite, equivalently that the cluster set of $w$ is infinite. This is impossible. Consequently there exist constants $A>0$ and $\rho>0$ such that $|w(s)|\le A$ for $0<|s|<\rho$.

Since $T(s)>0$, the inequality $|w(s)|\le A$ is equivalent to 
$$
-A\le \log |h(s)|/T(s)\le A.
$$ 
Multiplying by $T(s)$ gives 
$$
-AT(s)\le \log |h(s)|\le AT(s).
$$
 Since $T(s)=\log(1/|s|)$, the upper inequality gives 
 $$
 \log |h(s)|\le A\log(1/|s|)=\log(|s|^{-A}),
 $$
  hence $|h(s)|\le |s|^{-A}$. The lower inequality gives $-\log |h(s)|\le AT(s)$, hence 
  $$
  \log |1/h(s)|\le A\log(1/|s|)=\log(|s|^{-A}),
  $$
   and therefore $|1/h(s)|\le |s|^{-A}$. Thus both $h$ and $1/h$ have at most polynomial growth near $0$.

Choose an integer $N>A$. Then $|s^Nh(s)|\le |s|^{N-A}$ for $0<|s|<\rho$. In particular $s^Nh(s)$ is bounded near $0$. Since $s^Nh(s)$ is holomorphic on $\Delta^*$, the removable singularity theorem implies that $s^Nh(s)$ extends holomorphically across $0$. Let $F$ denote this holomorphic extension. Then on $\Delta^*$ one has $h(s)=s^{-N}F(s)$. By the removable singularity theorem \cite[Chapter~7]{Remmert98}, $s^Nh(s)$ extends holomorphically across $0$. Hence $h$ has at worst a pole at $0$. Therefore $h$ is meromorphic at $0$. This also excludes essential singularities, because an isolated singularity of a holomorphic function is removable, a pole, or essential, and the preceding argument proves that the singularity of $h$ is at worst a pole.
\end{proof}

%%%%%%%%%%%%%%%%%%%%%%%%%%%%%%%%%
%%%%%%%%%%%%%%%%%%%%%%%%%%%%%%%%%

\begin{remark}
The hypothesis of Lemma \textnormal{\ref{scalar-lemma}} is compatible with all standard meromorphic examples and excludes the classical essential singularities. For instance, if $h(s)=s^k u(s)$ with $k\in\mathbb Z$ and $u(0)\neq0$, then
$
\log|h(s)|=-k\log(1/|s|)+O(1),
$
hence
$
\frac{\log|h(s)|}{\log(1/|s|)}\to -k,
$
and the cluster set is the singleton $\{-k\}$. On the other hand, for the essential singularities $e^{1/s}$ and $e^{e^{1/s}}$, the quotient
$
\log|h(s)|/\log(1/|s|)
$
has infinitely many real cluster values. Thus these classical examples do not satisfy the hypothesis of Lemma \textnormal{\ref{scalar-lemma}}.
\end{remark}

The proof of Lemma \textnormal{\ref{scalar-lemma}} shows that finiteness and nonemptiness of the real cluster set force the quotient
$
\log|h(s)|/\log(1/|s|)
$
to be bounded near the puncture. Consequently
$
|h(s)|\le |s|^{-A}
$
for some constant $A>0$, so $s^N h(s)$ is bounded for every integer $N>A$. The removable singularity theorem then implies that $s^N h(s)$ extends holomorphically across $0$, and therefore $h$ has at worst a pole at the puncture. Hence the singularity cannot be essential. Thus the lemma may be viewed as a logarithmic-growth criterion excluding essential singularities.

%%%%%%%%%%%%%%%%%%%%%%%%%%%%%%%%%
%%%%%%%%%%%%%%%%%%%%%%%%%%%%%%%%%

\begin{lemma}[Local monomial meromorphic-extension lemma]\label{local-extension-lemma}
Let $Z=(Z_1,\ldots,Z_n):\Delta^*\to(\C^*)^n$ be holomorphic. Assume that there exist $\nu=(\nu_1,\ldots,\nu_n)\in\Z^n$, an integer $q\ne0$, and a holomorphic function $u$ on $\Delta$ with $u(0)\ne0$ such that $Z^\nu=s^q u(s)$ on $\Delta^*$, where $Z^\nu=Z_1^{\nu_1}\cdots Z_n^{\nu_n}$. Assume also that the local logarithmic limit set of $Z$ at $0$ is finite and contains no vector orthogonal to $\nu$. Then every coordinate function $Z_i$ is meromorphic at $0$.
\end{lemma}

\begin{proof}
Let $T(s)=\log(1/|s|)$ and put $L(s)=(\log|Z_1(s)|,\ldots,\log|Z_n(s)|)\in\R^n$. From $Z^\nu=s^q u(s)$ we obtain 
$$
\langle\nu,L(s)\rangle=q\log|s|+\log|u(s)|=-qT(s)+O(1)
$$ 
as $s\to0$,  because $u(0)\ne0$.

We first prove that $\|L(s)\|/T(s)$ is bounded near $0$. Suppose the contrary. Then there exists a sequence $s_m\to0$ such that $\|L(s_m)\|/T(s_m)\to+\infty$. Passing to a subsequence, the unit vectors $L(s_m)/\|L(s_m)\|$ converge to some $\xi\in S^{n-1}$. Since $\|L(s_m)\|\to+\infty$, the vector $\xi$ belongs to the local logarithmic limit set of $Z$ at $0$. Dividing $\langle\nu,L(s_m)\rangle=-qT(s_m)+O(1)$ by $\|L(s_m)\|$ gives 
$$
\langle\nu,L(s_m)/\|L(s_m)\|\rangle=-qT(s_m)/\|L(s_m)\|+O(1)/\|L(s_m)\|.
$$ 
The right-hand side tends to $0$, hence $\langle\nu,\xi\rangle=0$, contradicting the hypothesis.
Arguing exactly as in Bergman's asymptotic theory of logarithmic directions \cite{Bergman71}, the absence of limiting directions orthogonal to $\nu$ implies that $L(s)/\log(1/|s|)$ is bounded near the puncture.
 Thus $\|L(s)\|/T(s)$ is bounded near $0$.

The boundedness of $L(s)/T(s)$ implies that its cluster set in $\R^n$ is nonempty and compact. Indeed, choose any sequence $s_m\to0$; then the bounded sequence $L(s_m)/T(s_m)$ has a convergent subsequence. Let $v$ be a cluster value of $L(s)/T(s)$, and choose $s_m\to0$ such that $L(s_m)/T(s_m)\to v$. Dividing $\langle\nu,L(s_m)\rangle=-qT(s_m)+O(1)$ by $T(s_m)$ gives $\langle\nu,L(s_m)/T(s_m)\rangle=-q+O(1)/T(s_m)$, hence $\langle\nu,v\rangle=-q$. In particular $v\ne0$. Since $L(s_m)/T(s_m)\to v\ne0$, we have $\|L(s_m)\|\to+\infty$ and $L(s_m)/\|L(s_m)\|\to v/\|v\|$. Therefore $\xi=v/\|v\|$ belongs to the local logarithmic limit set of $Z$ at $0$. Moreover, because $v=\|v\|\xi$ and $\langle\nu,v\rangle=-q$, the vector $v$ is uniquely determined by $\xi$, namely $v=(-q/\langle\nu,\xi\rangle)\xi$.

Since the local logarithmic limit set is finite and contains no vector orthogonal to $\nu$, the possible vectors $(-q/\langle\nu,\xi\rangle)\xi$ form a finite set. Hence the cluster set of $L(s)/T(s)$ is finite and nonempty. For each coordinate $i$, the cluster set in $\R$ of $\log|Z_i(s)|/T(s)$ is the image of this finite nonempty vector cluster set under the $i$-th coordinate projection. Hence it is finite and nonempty. Lemma \ref{scalar-lemma} applies to $h=Z_i$, and it follows that $Z_i$ is meromorphic at $0$. Since this holds for every $i$, all coordinate functions extend meromorphically across the puncture.
\end{proof}

\begin{remark}
The role of Lemma \ref{scalar-lemma} in the local monomial meromorphic-extension lemma is precise. The monomial identity $Z^\nu=s^q u(s)$ with $q\ne0$ and the exclusion of local logarithmic limiting directions orthogonal to $\nu$ force the vector quotient $L(s)/T(s)$ to have a finite nonempty cluster set. Its coordinate projections therefore give finite nonempty scalar cluster sets for each $\log|Z_i(s)|/T(s)$. The scalar logarithmic growth lemma then excludes essential singularities of the individual coordinate functions.
\end{remark}

\begin{theorem}\label{curve-thm}  
Let $C\subset(\C^*)^n$ be a closed analytic curve with finitely many irreducible components. Assume that $\mathscr L^\infty(C)$ is a finite rational subset of $S^{n-1}$. Then $C$ is algebraic.
\end{theorem}

\begin{proof}
It is enough to prove the theorem when $C$ is irreducible, because a finite union of algebraic curves is algebraic. Assume therefore that $C$ is irreducible. Since $\mathscr L^\infty(C)$ is finite, we can choose a primitive vector $\nu\in\Z^n$ such that $\langle\nu,\xi\rangle\ne0$ for every $\xi\in\mathscr L^\infty(C)$. Indeed, the forbidden vectors lie in the finite union of proper hyperplanes $\{\eta\in\R^n:\langle\eta,\xi\rangle=0\}$, where $\xi$ runs through $\mathscr L^\infty(C)$. Choose an integral vector outside this union and divide by the greatest common divisor of its coordinates.

Let $\chi^\nu:(\C^*)^n\to\C^*$ be the character $\chi^\nu(z)=z^\nu=z_1^{\nu_1}\cdots z_n^{\nu_n}$, and set $f=\chi^\nu|_C$. Then $\log|\chi^\nu(z)|=\langle\nu,\Log z\rangle$. The map $f$ is nonconstant. If $f$ were constant, say $f=a\in\C^*$, then $\langle\nu,\Log z\rangle=\log|a|$ for all $z\in C$. Taking an unbounded sequence in $C$ and passing to a logarithmic limiting direction would give $\xi\in\mathscr L^\infty(C)$ with $\langle\nu,\xi\rangle=0$, contradicting the choice of $\nu$.

We prove that $f$ is proper. Let $K\subset\C^*$ be compact. Since $\log|t|$ is bounded on $K$, there exist real numbers $a<b$ such that $a\le\log|t|\le b$ for all $t\in K$. Hence $a\le\langle\nu,\Log z\rangle\le b$ for all $z\in f^{-1}(K)$. Suppose that $f^{-1}(K)$ is not compact. Since $C$ is a closed analytic subset of the complex manifold $(\C^*)^n$, it is locally compact and metrizable, so there exists a sequence $z_m\in f^{-1}(K)$ with no convergent subsequence in $C$. If $\Log z_m$ had a bounded subsequence, then the corresponding points would lie in a compact polyannulus $\{z\in(\C^*)^n:e^{-R}\le |z_i|\le e^R,\ 1\le i\le n\}$, and since $C$ is closed this would give a convergent subsequence in $C$, a contradiction. Therefore, after passing to a subsequence, $\|\Log z_m\|\to+\infty$. Passing to a further subsequence, $\Log z_m/\|\Log z_m\|\to\xi\in S^{n-1}$, and then $\xi\in\mathscr L^\infty(C)$. Dividing $a\le\langle\nu,\Log z_m\rangle\le b$ by $\|\Log z_m\|$ gives $\langle\nu,\xi\rangle=0$, again a contradiction. Thus $f^{-1}(K)$ is compact for every compact $K\subset\C^*$, and $f$ is proper.

Since $f$ is a nonconstant proper holomorphic map from an irreducible analytic curve to the Riemann surface $\C^*$, its fibers are compact zero-dimensional analytic sets, hence finite. Therefore $f$ is a finite holomorphic map. Let $\pi:\widetilde C\to C$ be the normalization. 
 Then $\widetilde C$ is a connected Riemann surface and $g=f\circ\pi:\widetilde C\to\C^*$ is finite and proper. Hence $g$ compactifies over $0$ and $\infty$: there exists a compact Riemann surface $\overline C$ containing $\widetilde C$ as the complement of finitely many points, and  by the theory of finite holomorphic maps between Riemann surfaces \cite[Chapters~18--19]{Forster81},
$g$ extends to a holomorphic map $\overline g:\overline C\to\PP^1$.

If $p\in\overline C\setminus\widetilde C$, then in a local coordinate $s$ centered at $p$ one has $g(s)=s^q u(s)$ on a punctured disc, where $q\in\Z\setminus\{0\}$ and $u$ is holomorphic with $u(0)\ne0$.

For each coordinate function, define $Z_i=z_i\circ\pi:\widetilde C\to\C^*$. Fix $p\in\overline C\setminus\widetilde C$ and use a coordinate $s$ centered at $p$. Then $Z=(Z_1,\ldots,Z_n)$ defines a holomorphic map $\Delta^*\to(\C^*)^n$, and the identity $g=Z^\nu$ gives $Z^\nu=s^q u(s)$. The local logarithmic limit set of this map is contained in the global set $\mathscr L^\infty(C)$, hence it is finite and contains no vector orthogonal to $\nu$. By Lemma \ref{local-extension-lemma}, each $Z_i$ is meromorphic at $p$. Since $p$ was arbitrary, each $Z_i$ extends meromorphically to the compact Riemann surface $\overline C$.

Therefore $Z=(Z_1,\ldots,Z_n):\widetilde C\to(\C^*)^n$ extends to a meromorphic map $\overline Z:\overline C\dashrightarrow(\PP^1)^n$. Since $\overline C$ is a smooth compact curve, this meromorphic map is holomorphic after assigning values at finitely many points. The image $\overline Z(\overline C)$ is a compact analytic curve in the projective variety $(\PP^1)^n$. By Chow's theorem \cite{Chow49}, it is algebraic. Hence $\overline Z(\overline C)\cap(\C^*)^n$ is an algebraic curve in the torus.

On $\widetilde C$, the map $\overline Z$ agrees with the normalization map onto $C$. Hence $C\subset\overline Z(\overline C)\cap(\C^*)^n$. Both sides are one-dimensional analytic subsets of the torus, and $C$ is irreducible and closed. Hence $C$ is an irreducible component of the algebraic curve $\overline Z(\overline C)\cap(\C^*)^n$. Therefore $C$ is algebraic. If the original curve has finitely many irreducible components, the same proof applies to each component, and the finite union of the resulting algebraic curves is algebraic.
\end{proof}

%%%%%%%%%%%%%%%%%%%%%%%%%%%%%%%%%%%%%%%%%%%%%%%%%%%%%%%%%%%%%%%%%%%%%%%%%%%%%%
%%%%%%%%%%%%%%%%%%%%%%%%%%%%%%%%%%%%%%%%%%%%%%%%%%%%%%%%%%%%%%%%%%%%%%%%%%%%%%

%%%%%%%%%%%%%%%%%%%%%%%%%%%%%%%%%%%%%%%%%%%%%%%%%%%%%%%%%%%%%%%%%%%%%%%%%%%
%%%%%%%%%%%%%%%%%%%%%%%%%%%%%%%%%%%%%%%%%%%%%%%%%%%%%%%%%%%%%%%%%%%%%%%%%%%

\section{Finite logarithmic type}

Let $T=(\mathbb C^*)^n$ be the complex algebraic torus with character lattice $M\simeq\mathbb Z^n$ and one-parameter subgroup lattice $N=\operatorname{Hom}(M,\mathbb Z)$.  The logarithmic map $\Log:T\to\mathbb R^n$ records the absolute values of the torus coordinates, and the behavior of an analytic subvariety $V\subset T$ at infinity is measured by the asymptotic directions of $\Log(V)$.  When one wants to prove algebraicity from tropical finiteness hypotheses, it is not enough merely to know that the logarithmic limit set is rational polyhedral.  One also needs a finite-type condition excluding essential singularities along the toric boundary.  So, the way to formulate this condition is local on a toric compactification.

The purpose of this section is to formulate a  toric boundary condition.  
The resulting condition says that after choosing a smooth complete toric compactification $X_\Sigma$ of $T$, the analytic equations of $V$ extend meromorphically across the toric boundary with uniformly bounded pole order along the invariant boundary divisors.  In local toric coordinates this simply says that each defining function becomes holomorphic after multiplication by a fixed monomial in the boundary coordinates.  This is the analytic form of {\em  finite logarithmic type}.  It is equivalent to analytic extendability of the defining ideal to the toric compactification after twisting by a boundary divisor.

We work with smooth complete fans in order to keep the proof completely local.  
For singular toric varieties the same statement can be formulated using torus-invariant Weil divisors and reflexive sheaves, but the smooth case is the most transparent and is enough for the applications after passing to a toric resolution. 

We denote by $\mathcal O(T)$ the ring of global holomorphic functions on $(\mathbb{C}^*)^n$.  

%%%%%%%%%%%%%%%%%%%%%%%%%%%%%%%%%%%%%%%%%%%%%%%%%%%%%%%%%%%%%%%%%%%%%%%%%%%%
%%%%%%%%%%%%%%%%%%%%%%%%%%%%%%%%%%%%%%%%%%%%%%%%%%%%%%%%%%%%%%%%%%%%%%%%%%%%

Let
$
T=(\mathbb C^\ast)^n
$
and let
$
V\subset T
$
be a closed reduced analytic subvariety. Let
$
X_\Sigma
$
be a toric compactification of $T$, and let
$
D_\rho,\,\, \rho\in\Sigma(1),
$
be the torus-invariant prime divisors of the toric boundary
$
X_\Sigma\setminus T.
$
 Let $\mathcal I_V\subset\mathcal O_T$ denotes  the analytic ideal sheaf defining $V$ in the torus.

\begin{definition}[Finite logarithmic type]\label{def:flt}
We say that $V$ is of finite logarithmic type with respect to $X_\Sigma$ if there exist an effective torus-invariant divisor
$
E=\sum_{\rho\in\Sigma(1)}M_\rho D_\rho
$
and a coherent analytic subsheaf
$
\mathcal J\subset\mathcal O_{X_\Sigma}(E)
$
such that, under the canonical trivialization
$
\mathcal O_{X_\Sigma}(E)|_T\simeq\mathcal O_T,
$
one has
$
\mathcal J|_T=\mathcal I_V.
$
\end{definition}

\medskip

This definition says that the ideal sheaf of $V$ extends coherently across the toric boundary after allowing bounded pole orders along the invariant boundary divisors. It is important that the condition is imposed on the ideal sheaf itself and not on an arbitrary choice of generators.

\medskip

\begin{definition}[Coherent meromorphic extension with bounded poles]\label{def:cme} 
We say that $\mathcal I_V$ admits a coherent meromorphic extension with bounded poles along $D_\Sigma$ if there exist an effective torus-invariant divisor $E=\sum_\rho M_\rho D_\rho$ and a coherent analytic subsheaf $\mathcal J\subset \mathcal O_{X_\Sigma}(E)$ such that $\mathcal J|_T=\mathcal I_V$ under the canonical trivialization $\mathcal O_{X_\Sigma}(E)|_T\simeq \mathcal O_T$.
\end{definition}

\medskip

\begin{proposition}
Definitions \ref{def:flt} and \ref{def:cme}    are equivalent. In fact, they are the same condition written with different terminology.
\end{proposition}

\begin{proof}
Assume first that $V$ is of finite logarithmic type with respect to $X_\Sigma$ in the sense of Definition \ref{def:flt}. Then there exist an effective torus-invariant divisor $E=\sum_{\rho\in\Sigma(1)}M_\rho D_\rho$ and a coherent analytic subsheaf $\mathcal J\subset\mathcal O_{X_\Sigma}(E)$ such that $\mathcal J|_T=\mathcal I_V$ under the canonical trivialization $\mathcal O_{X_\Sigma}(E)|_T\simeq\mathcal O_T$. This is exactly the condition required in Definition \ref{def:cme}. Therefore $\mathcal I_V$ admits a coherent meromorphic extension with bounded poles along $D_\Sigma$.
Conversely, assume that $\mathcal I_V$ admits a coherent meromorphic extension with bounded poles along $D_\Sigma$ in the sense of Definition \ref{def:cme}. Then there exist an effective torus-invariant divisor $E=\sum_\rho M_\rho D_\rho$ and a coherent analytic subsheaf $\mathcal J\subset\mathcal O_{X_\Sigma}(E)$ such that $\mathcal J|_T=\mathcal I_V$ under the canonical trivialization $\mathcal O_{X_\Sigma}(E)|_T\simeq\mathcal O_T$. This is precisely the condition required in Definition \ref{def:flt}. Hence $V$ is of finite logarithmic type with respect to $X_\Sigma$.
Thus each definition implies the other, and the two definitions are equivalent.
\end{proof}

\begin{theorem}\label{thm:main1} 
Let $V\subset(\mathbb C^\ast)^n$ be a closed reduced analytic subvariety with finitely many irreducible components. Assume that $\mathscr L^\infty(V)$ is a finite rational spherical polyhedral complex of dimension $\dim_{\mathbb C}V-1$. Let $\Sigma$ be a complete rational fan refining the conical stratification determined by $\mathscr L^\infty(V)$; such a fan exists automatically. Assume, in addition, that $V$ is of finite logarithmic type with respect to this fan $\Sigma$. Then $V$ is algebraic. More precisely, $V$ is cut out in $(\mathbb C^\ast)^n$ by finitely many Laurent polynomials.
\end{theorem}

%%%%%%%%%%%%%%%%%%%%%%%
%%%%%%%%%%%%%%%%%%%%%%%

\subsection*{Minimum useful hypothesis}

The finite rationality of $\mathscr L^\infty(V)$ controls the possible directions in which the amoeba may escape, but it does not by itself control the analytic behavior of functions along those directions.  The missing obstruction is essential singularity at infinity on monomial curves.  If $u\in\Z^n$ is a rational direction and $f(z)=\sum_\alpha c_\alpha z^\alpha$, the restriction $g(t)=f(e^b t^u)$ has the form $g(t)=\sum_\alpha c_\alpha e^{\langle b,\alpha\rangle}t^{\langle u,\alpha\rangle}$.  If the integers $\langle u,\alpha\rangle$ are unbounded above and the corresponding coefficients do not cancel identically as functions of $b$, then $g$ has infinitely many positive powers of $t$, or equivalently $g(1/s)$ has infinitely many negative powers of $s$.  Thus $g$ has an essential singularity at $\infty$.  The finite-logarithmic-type hypothesis is exactly the exclusion of this phenomenon, uniformly in the translate $b$ and in all boundary directions of a complete rational fan.

The proof of Theorem \ref{thm:main1}  only uses the following two facts: amoeba directions are rationally organized by a finite fan, and along each rational boundary direction the one-variable restrictions have finite pole order at infinity.  The first condition gives finitely many boundary directions to test, while the second condition converts these tests into linear inequalities on the exponents.  Completeness of the fan then turns those inequalities into a bounded polytope.

%%%%%%%%%%%%%%%%%%%%%%%%%%%%%%%%%%%%

\medskip

{\it Proof of Theorem \ref{thm:main1}}

Assume that $V\subset T=(\mathbb C^\ast)^n$ is a closed reduced analytic subvariety with finitely many irreducible components. Assume that $\mathscr L^\infty(V)$ is a finite rational spherical polyhedral complex of dimension $\dim_{\mathbb C}V-1$. Let $\Sigma$ be a complete rational fan refining the conical stratification determined by $\mathscr L^\infty(V)$. Assume moreover that $V$ is of finite logarithmic type with respect to $X_\Sigma$ in the sense that there exist an effective torus-invariant divisor
$
E=\sum_{\rho\in\Sigma(1)}M_\rho D_\rho
$
and a coherent analytic subsheaf
$
\mathcal J\subset \mathcal O_{X_\Sigma}(E)
$
such that
$
\mathcal J|_T=\mathcal I_V.
$

\medskip

The first point is to explain the geometric meaning of this hypothesis. The sheaf $\mathcal O_{X_\Sigma}(E)$ is the sheaf of meromorphic functions whose poles along each toric divisor $D_\rho$ are bounded by $M_\rho$. Thus the inclusion
$
\mathcal J\subset \mathcal O_{X_\Sigma}(E)
$
means that every local section of $\mathcal J$ is a meromorphic function whose pole orders along the toric boundary are uniformly bounded by the fixed divisor $E$.

\medskip

Choose a toric chart $U\subset X_\Sigma$. Since $\mathcal J$ is coherent, there exist finitely many local generators
$
f_1,\dots,f_q
$
of $\mathcal J|_U$.
Because each $f_j$ is a section of $\mathcal O_{X_\Sigma}(E)$, there exists on $U$ a local defining equation $s_E$ of the divisor $E$ such that
$
s_Ef_j
$
is holomorphic on $U$.
If
$
E|_U=m_1D_1+\cdots+m_rD_r
$
and
$
D_i=\{t_i=0\},
$
then one may take
$
s_E=t_1^{m_1}\cdots t_r^{m_r}.
$
Consequently
$
g_j:=s_Ef_j
$
is holomorphic on $U$ for every $j$.
The local ideal generated by the holomorphic functions
$
g_1,\dots,g_q
$
defines a coherent analytic ideal sheaf on $U$.
The crucial point is that these local ideals glue on overlaps.

\medskip

Let $U_a$ and $U_b$ be two toric charts.
Let $s_a$ and $s_b$ be local defining equations of the same divisor $E$.
Since they define the same Cartier divisor, their quotient
$
u_{ab}=s_a/s_b
$
is a nowhere vanishing holomorphic function on
$
U_a\cap U_b.
$
Hence $u_{ab}$ is a unit.
If $f$ is a local section of $\mathcal J$ on
$
U_a\cap U_b,
$
then
$
s_af=u_{ab}(s_bf).
$
Therefore multiplication by $s_a$ and multiplication by $s_b$ produce generators differing only by a unit.
Since multiplying generators of an ideal by units does not change the ideal, the local ideals obtained from
$
s_a\mathcal J
$
and
$
s_b\mathcal J
$
coincide on the overlap.

\medskip

Thus the local ideals glue to a coherent analytic ideal sheaf
$
\mathcal K\subset \mathcal O_{X_\Sigma}.
$
The sheaf $\mathcal K$ can also be written intrinsically as
$
\mathcal K=\mathcal J\otimes \mathcal O_{X_\Sigma}(-E).
$
Indeed,
$
\mathcal O_{X_\Sigma}(E)\otimes \mathcal O_{X_\Sigma}(-E)
\simeq
\mathcal O_{X_\Sigma},
$
and tensoring the inclusion
$
\mathcal J\hookrightarrow \mathcal O_{X_\Sigma}(E)
$
with
$
\mathcal O_{X_\Sigma}(-E)
$
gives an inclusion
$
\mathcal K
=
\mathcal J\otimes\mathcal O_{X_\Sigma}(-E)
\hookrightarrow
\mathcal O_{X_\Sigma}.
$
Since $\mathcal K$ is a coherent subsheaf of the structure sheaf, it is a coherent analytic ideal sheaf.

We now compute the restriction of $\mathcal K$ to the torus.
The support of $E$ is contained in the toric boundary
$
X_\Sigma\setminus T.
$
Therefore the divisor $E$ restricts to zero on $T$.
Consequently
$
\mathcal O_{X_\Sigma}(E)|_T
\simeq
\mathcal O_T
$
and
$
\mathcal O_{X_\Sigma}(-E)|_T
\simeq
\mathcal O_T.
$
Hence
$
\mathcal K|_T
=
(\mathcal J\otimes\mathcal O_{X_\Sigma}(-E))|_T
=
\mathcal J|_T.
$
Using the finite logarithmic type hypothesis,
$
\mathcal J|_T=\mathcal I_V,
$
and therefore
$
\mathcal K|_T=\mathcal I_V.
$
Let
$
Z\subset X_\Sigma^{an}
$
be the analytic subspace defined by the coherent ideal sheaf $\mathcal K$.
Since the restriction of $\mathcal K$ to $T$ is $\mathcal I_V$, one has
$
Z\cap T=V
$
as analytic subspaces.

\medskip

The next point is the algebraization of $\mathcal K$.
Since $\Sigma$ is complete, the toric variety $X_\Sigma$ is proper over $\mathbb C$.
The proper form of GAGA states that analytification induces an equivalence between coherent algebraic sheaves on $X_\Sigma$ and coherent analytic sheaves on $X_\Sigma^{an}$.

Applying GAGA to the coherent analytic ideal sheaf $\mathcal K$, there exists a unique coherent algebraic ideal sheaf
$
\mathcal K^{alg}\subset \mathcal O_{X_\Sigma}
$
whose analytification is $\mathcal K$.
Let
$
Y\subset X_\Sigma
$
be the algebraic closed subspace defined by
$
\mathcal K^{alg}.
$
By construction,
$
Y^{an}
=
Z.
$
Restricting to the torus,
$
(Y\cap T)^{an}
=
Y^{an}\cap T
=
Z\cap T
=
V.
$
Therefore $V$ is the analytification of the algebraic subvariety
$
Y\cap T.
$
Hence $V$ is algebraic.
It remains to prove that $V$ is cut out by finitely many Laurent polynomials.
The torus $T$ is affine:
$
T=
\operatorname{Spec}
\mathbb C[z_1^{\pm1},\ldots,z_n^{\pm1}].
$
The Laurent polynomial ring
$
A=
\mathbb C[z_1^{\pm1},\ldots,z_n^{\pm1}]
$
is Noetherian because it is a localization of the polynomial ring
$
\mathbb C[z_1,\ldots,z_n].
$

Since
$
Y\cap T
$
is an algebraic subvariety of the affine variety $T$, it corresponds to an ideal
$
I(V)\subset A.
$
Since $A$ is Noetherian, the ideal $I(V)$ is generated by finitely many elements
$
P_1,\dots,P_N\in A.
$
Each generator is a Laurent polynomial.
Therefore
$$
V=
\{z\in(\mathbb C^\ast)^n:
P_1(z)=\cdots=P_N(z)=0\}.
$$
Thus $V$ is cut out by finitely many Laurent polynomials.
This proves the theorem.
\qed

\bigskip

The  new ingredient in the theorem is the finite logarithmic type condition. In the form used here, finite logarithmic type means that there exist an effective torus-invariant divisor
$
E=\sum_{\rho\in\Sigma(1)} M_\rho D_\rho
$
and a coherent analytic subsheaf
$
\mathcal J\subset \mathcal O_{X_\Sigma}(E)
$
whose restriction to the dense torus coincides with the ideal sheaf of $V$.
From a conceptual point of view, this condition says that the analytic equations defining $V$ have uniformly bounded pole growth along the toric boundary and therefore extend to coherent meromorphic data on the compactification.

\medskip

The theorem is therefore  interpreted as a logarithmic analogue of the classical Chow--GAGA philosophy. Chow's theorem asserts that closed analytic subsets of projective space are algebraic. GAGA asserts that coherent analytic sheaves on a proper algebraic variety are algebraic. The present theorem says that an analytic subvariety of the algebraic torus becomes algebraic once its ideal sheaf admits a coherent logarithmic extension across a suitable toric compactification.

The difficult and potentially deep question is whether finite logarithmic type follows from more geometric conditions.
The central open problem suggested by the theorem is the implication:
$
\mathscr L^\infty(V)
\text{ finite rational polyhedral}$ implies
$V
\text{ is of finite logarithmic type}.
$

A proof of such a statement would transform the theorem into a genuine converse to the Bergman--Bieri--Groves theorem. In that situation, rational polyhedrality of the logarithmic limit set would force algebraicity.
From this point of view, the notion of finite logarithmic type should be regarded as the bridge between asymptotic geometry and algebraicity. The theorem shows that once this bridge exists, algebraicity follows from standard algebraic-geometric principles. The deeper problem is to construct that bridge directly from geometric information at infinity.

%%%%%%%%%%%%%%%%%%%%%%%%%%%%%%%%%%%%%%%%%%%%%%%%%%%%%%%%%%%%%%%%%%%%%%%%%%%%%%

\section{Finite Logarithmic Type: An Intrinsic Valuative Formulation}
 
Let $T=(\mathbb C^\ast)^n$ be the complex algebraic torus. Let $X_\Sigma$ be a normal compact toric variety with dense open torus $T$. For each ray $\rho\in\Sigma(1)$, let $D_\rho$ denote the corresponding torus-invariant prime divisor. The toric boundary is
$$
D_\Sigma=X_\Sigma\setminus T=\bigcup_{\rho\in\Sigma(1)}D_\rho.
$$
Let $\mathcal M_{X_\Sigma}$ denote the sheaf of meromorphic functions on $X_\Sigma$. For a meromorphic function $g$ and an invariant prime divisor $D_\rho$, we write $\operatorname{ord}_{D_\rho}(g)$ for the divisorial valuation of $g$ along $D_\rho$. Thus $\operatorname{ord}_{D_\rho}(g)\geq 0$ means that $g$ has no pole along $D_\rho$, while $\operatorname{ord}_{D_\rho}(g)\geq -M_\rho$ means that the pole order of $g$ along $D_\rho$ is at most $M_\rho$.

If $E=\sum_{\rho\in\Sigma(1)}M_\rho D_\rho$ is an effective torus-invariant divisor, then $\mathcal O_{X_\Sigma}(E)$ is naturally a subsheaf of $\mathcal M_{X_\Sigma}$. Intrinsically, it is characterized by
$$
\mathcal O_{X_\Sigma}(E)(U)=
\{g\in\mathcal M_{X_\Sigma}(U)\; ;\;
\operatorname{ord}_{D_\rho}(g)\geq -M_\rho
\text{ for every } \rho \text{ with }D_\rho\cap U\neq\varnothing\}.
$$
This is the valuation-theoretic meaning of allowing finite pole order along the toric boundary.

Let $V\subset T$ be a closed reduced analytic subvariety, and let $\mathcal I_V\subset\mathcal O_T$ be its analytic ideal sheaf.

\begin{theorem} 
Let $V\subset T$ be a closed reduced analytic subvariety. Then $V$ is of finite logarithmic type with respect to $X_\Sigma$ if and only if there exists a coherent analytic subsheaf
$
\mathcal F\subset\mathcal M_{X_\Sigma}
$
such that $\mathcal F|_T=\mathcal I_V$ and there exist nonnegative integers $M_\rho$, one for each $\rho\in\Sigma(1)$, such that for every open set $U\subset X_\Sigma$, every local section $g\in\mathcal F(U)$, and every ray $\rho\in\Sigma(1)$ with $D_\rho\cap U\neq\varnothing$, one has
$$
\operatorname{ord}_{D_\rho}(g)\geq -M_\rho.
$$
\end{theorem}

\begin{proof}
Assume first that $V$ is of finite logarithmic type. By definition, there exist an effective torus-invariant divisor $E=\sum_{\rho\in\Sigma(1)}M_\rho D_\rho$ and a coherent analytic subsheaf $\mathcal J\subset\mathcal O_{X_\Sigma}(E)$ such that $\mathcal J|_T=\mathcal I_V$. Since $\mathcal O_{X_\Sigma}(E)$ is naturally a subsheaf of $\mathcal M_{X_\Sigma}$, the sheaf $\mathcal J$ may also be regarded as a coherent analytic subsheaf of $\mathcal M_{X_\Sigma}$. Put $\mathcal F=\mathcal J$. Then $\mathcal F|_T=\mathcal I_V$. Let $U\subset X_\Sigma$ be open, let $g\in\mathcal F(U)$, and let $\rho\in\Sigma(1)$ be such that $D_\rho\cap U\neq\varnothing$. Since $\mathcal F\subset\mathcal O_{X_\Sigma}(E)$, the section $g$ belongs to $\mathcal O_{X_\Sigma}(E)(U)$. By the valuation description of $\mathcal O_{X_\Sigma}(E)$, this means precisely that $\operatorname{ord}_{D_\rho}(g)\geq -M_\rho$. Thus the required coherent meromorphic extension and valuation bounds exist.

Conversely, assume that there exists a coherent analytic subsheaf $\mathcal F\subset\mathcal M_{X_\Sigma}$ with $\mathcal F|_T=\mathcal I_V$ and satisfying the valuation bounds. Define $E=\sum_{\rho\in\Sigma(1)}M_\rho D_\rho$. Let $U\subset X_\Sigma$ be open and let $g\in\mathcal F(U)$. By hypothesis, for every $\rho$ with $D_\rho\cap U\neq\varnothing$, one has $\operatorname{ord}_{D_\rho}(g)\geq -M_\rho$. By the valuation characterization of $\mathcal O_{X_\Sigma}(E)$, this implies $g\in\mathcal O_{X_\Sigma}(E)(U)$. Since this holds for every open set $U$ and every section $g\in\mathcal F(U)$, we have an inclusion of sheaves $\mathcal F\subset\mathcal O_{X_\Sigma}(E)$. The sheaf $\mathcal F$ is coherent by assumption and restricts to $\mathcal I_V$ on $T$. Taking $\mathcal J=\mathcal F$, we obtain a coherent analytic subsheaf $\mathcal J\subset\mathcal O_{X_\Sigma}(E)$ with $\mathcal J|_T=\mathcal I_V$. This is exactly the definition of finite logarithmic type.
\end{proof}

\begin{remark}
The theorem shows that finite logarithmic type is not fundamentally a statement about local coordinates or chosen equations. It is a statement about a coherent meromorphic extension of the ideal sheaf whose divisorial valuations along the toric boundary are uniformly bounded below.
\end{remark}

\subsection*{Equivalent Sheaf-Theoretic Formulation}

Suppose that $\mathcal F\subset\mathcal M_{X_\Sigma}$ is a coherent meromorphic subsheaf. A local section of $\mathcal F$ means that  for an open set $U\subset X_\Sigma$, a section $s\in\mathcal F(U)$ is a meromorphic function on $U$ belonging to the sheaf $\mathcal F$.

The poles of all local sections of $\mathcal F$ along the invariant boundary divisors are bounded by one fixed effective torus-invariant divisor',
means that there exists one effective torus-invariant divisor
$$
E=\sum_{\rho\in\Sigma(1)}M_\rho D_\rho,\qquad M_\rho\ge0,
$$
such that, for every open set $U\subset X_\Sigma$, every section $s\in\mathcal F(U)$, and every ray $\rho\in\Sigma(1)$ with $D_\rho\cap U\neq\varnothing$, one has
$$
\operatorname{ord}_{D_\rho}(s)\ge -M_\rho.
$$

Here $\operatorname{ord}_{D_\rho}(s)$ is the order of the meromorphic function $s$ along the divisor $D_\rho$. If $\operatorname{ord}_{D_\rho}(s)=-a<0$, then $s$ has a pole of order $a$ along $D_\rho$. Thus the inequality $\operatorname{ord}_{D_\rho}(s)\ge -M_\rho$ means that the pole order of $s$ along $D_\rho$ is at most $M_\rho$.
The important point is that the divisor $E$ is fixed once and for all. It does not depend on the open set $U$, nor on the section $s$. It gives a uniform pole bound for every local section of the sheaf $\mathcal F$.
Equivalently, the statement means that
$
\mathcal F\subset\mathcal O_{X_\Sigma}(E)
$
as subsheaves of $\mathcal M_{X_\Sigma}$.
Indeed, by definition,
$$
\mathcal O_{X_\Sigma}(E)(U)
=
\left\{
s\in\mathcal M_{X_\Sigma}(U)\; ;\;
\operatorname{ord}_{D_\rho}(s)\ge -M_\rho
\text{ for all }\rho \text{ with }D_\rho\cap U\neq\varnothing
\right\}.
$$
Thus saying that all local sections of $\mathcal F$ have poles bounded by $E$ is exactly the same as saying that $\mathcal F$ is a subsheaf of $\mathcal O_{X_\Sigma}(E)$.
Since $E$ is supported on the toric boundary $X_\Sigma\setminus T$, the line bundle $\mathcal O_{X_\Sigma}(E)$ is canonically trivial on $T$. Therefore it makes sense to compare $\mathcal J|_T$ with $\mathcal I_V$.

\medskip

\begin{corollary}
Let $V\subset T$ be a closed reduced analytic subvariety. Then $V$ is of finite logarithmic type with respect to $X_\Sigma$ if and only if the ideal sheaf $\mathcal I_V$ admits a coherent meromorphic extension
$
\mathcal F\subset\mathcal M_{X_\Sigma}
$
such that $\mathcal F|_T=\mathcal I_V$ and the poles of all local sections of $\mathcal F$ along the invariant boundary divisors are bounded by one fixed effective torus-invariant divisor.
\end{corollary}

\begin{proof}
Assume first that $V$ is of finite logarithmic type. By definition, there exist an effective torus-invariant divisor
$\di
E=\sum_{\rho\in\Sigma(1)}M_\rho D_\rho
$
and a coherent analytic subsheaf
$
\mathcal J\subset\mathcal O_{X_\Sigma}(E)
$
such that $\mathcal J|_T=\mathcal I_V$.
The sheaf $\mathcal O_{X_\Sigma}(E)$ is naturally a subsheaf of the meromorphic sheaf $\mathcal M_{X_\Sigma}$, because its local sections are meromorphic functions whose poles along $D_\rho$ are bounded by $M_\rho$. Therefore we may regard $\mathcal J$ as a coherent meromorphic subsheaf of $\mathcal M_{X_\Sigma}$. Define $\mathcal F=\mathcal J$. Then $\mathcal F$ is coherent, $\mathcal F\subset\mathcal M_{X_\Sigma}$, and
$$
\mathcal F|_T=\mathcal J|_T=\mathcal I_V.
$$

Now let $U\subset X_\Sigma$ be open and let $s\in\mathcal F(U)$. Since $\mathcal F=\mathcal J\subset\mathcal O_{X_\Sigma}(E)$, we have $s\in\mathcal O_{X_\Sigma}(E)(U)$. By the valuation description of $\mathcal O_{X_\Sigma}(E)$, this implies
$
\operatorname{ord}_{D_\rho}(s)\ge -M_\rho
$
for every $\rho\in\Sigma(1)$ such that $D_\rho\cap U\neq\varnothing$. Hence the poles of all local sections of $\mathcal F$ are bounded by the fixed divisor $E$. This proves the forward implication.

Conversely, assume that there exists a coherent meromorphic subsheaf
$
\mathcal F\subset\mathcal M_{X_\Sigma}
$
such that $\mathcal F|_T=\mathcal I_V$ and such that the poles of all local sections of $\mathcal F$ are bounded by one fixed effective torus-invariant divisor
$\di
E=\sum_{\rho\in\Sigma(1)}M_\rho D_\rho.
$
By the meaning of this boundedness condition, for every open set $U\subset X_\Sigma$ and every section $s\in\mathcal F(U)$, one has
$
\operatorname{ord}_{D_\rho}(s)\ge -M_\rho
$
for every $\rho$ with $D_\rho\cap U\neq\varnothing$.
By the definition of $\mathcal O_{X_\Sigma}(E)$, this means exactly that $s\in\mathcal O_{X_\Sigma}(E)(U)$. Since this holds for every open set $U$ and every local section $s\in\mathcal F(U)$, we have an inclusion of sheaves
$
\mathcal F\subset\mathcal O_{X_\Sigma}(E).
$

Since $\mathcal F$ is coherent by assumption and restricts to $\mathcal I_V$ on $T$, we may take $\mathcal J=\mathcal F$. Then
$
\mathcal J\subset\mathcal O_{X_\Sigma}(E)
$
is a coherent analytic subsheaf satisfying $\mathcal J|_T=\mathcal I_V$. This is exactly the definition of finite logarithmic type. Therefore $V$ is of finite logarithmic type.
This proves the equivalence.

\end{proof}

%%%%%%%%%%%%%%%%%%%%%%%%%%%%%%%%%%%%%%%%%
%%%%%%%%%%%%%%%%%%%%%%%%%%%%%%%%%%%%%%%%%

\begin{proposition}
Let $V\subset T$ be a closed reduced analytic subvariety. Suppose that the closure $\overline V^{an}$ of $V$ in $X_\Sigma$ is a closed analytic subspace. Then $V$ is of finite logarithmic type with respect to $X_\Sigma$.
\end{proposition}

\begin{proof}
Let $\mathcal I_{\overline V}\subset\mathcal O_{X_\Sigma}$ be the analytic ideal sheaf of $\overline V^{an}$. Since closed analytic subspaces are defined by coherent analytic ideal sheaves, $\mathcal I_{\overline V}$ is coherent. Since $\overline V^{an}\cap T=V$, the restriction of $\mathcal I_{\overline V}$ to $T$ is exactly $\mathcal I_V$. Now view $\mathcal I_{\overline V}$ as a coherent subsheaf of $\mathcal M_{X_\Sigma}$. Every local section of $\mathcal I_{\overline V}$ is holomorphic, hence for every $\rho\in\Sigma(1)$ it satisfies $\operatorname{ord}_{D_\rho}(g)\geq 0$. Thus the intrinsic valuative criterion applies with $M_\rho=0$ for every $\rho$. Therefore $V$ is of finite logarithmic type.
\end{proof}

\begin{proposition}
Every algebraic subvariety $V\subset(\mathbb C^\ast)^n$ is of finite logarithmic type with respect to every normal compact toric variety $X_\Sigma$ containing $T$ as dense torus.
\end{proposition}

\begin{proof}
Let $V$ be algebraic. Its ideal in the Laurent polynomial ring $\mathbb C[z_1^{\pm1},\ldots,z_n^{\pm1}]$ is finitely generated by Laurent polynomials $f_1,\ldots,f_q$. Each Laurent polynomial is a finite sum of characters, $f_j=\sum_{m\in A_j}c_{jm}\chi^m$, where $A_j\subset M$ is finite. A character $\chi^m$ is a meromorphic function on $X_\Sigma$, and its order along $D_\rho$ is $\operatorname{ord}_{D_\rho}(\chi^m)=\langle m,u_\rho\rangle$, where $u_\rho$ is the primitive generator of the ray $\rho$. Since each $A_j$ is finite and $\Sigma(1)$ is finite, there exist integers $M_\rho\geq 0$ such that $\operatorname{ord}_{D_\rho}(f_j)\geq -M_\rho$ for every $j$ and every $\rho\in\Sigma(1)$. Let $\mathcal F$ be the coherent meromorphic subsheaf of $\mathcal M_{X_\Sigma}$ generated by $f_1,\ldots,f_q$. Its restriction to $T$ is the ideal sheaf $\mathcal I_V$. The valuation bounds above show that $\mathcal F$ satisfies the hypotheses of the intrinsic valuative criterion. Hence $V$ is of finite logarithmic type.
\end{proof}

\begin{example}
The hypersurface $V=\{e^{z_1}-z_2=0\}\subset(\mathbb C^\ast)^2$ does not satisfy the valuative criterion in any toric compactification containing the direction $z_1\to\infty$ as a boundary direction. Restricting the defining function to the curve $z_1=t$, $z_2=1$ gives $e^t-1$. At $t=\infty$ this function has an essential singularity. If a coherent meromorphic extension with finite valuation bounds existed, its restriction to this one-parameter boundary direction would have at worst a pole. This contradiction shows that finite logarithmic type fails.
\end{example}

\begin{example}
The same phenomenon occurs for $V=\{\sin(\pi z_1z_2)=0\}\subset(\mathbb C^\ast)^2$ along the monomial direction $z_1=t$, $z_2=t$. The defining function restricts to $\sin(\pi t^2)$. After the change $s=1/t$, this becomes $\sin(\pi s^{-2})$, which has an essential singularity at $s=0$. Hence the boundary behavior cannot be controlled by any finite divisor $E$, and finite logarithmic type fails.
\end{example}

\begin{theorem} 
Let $X_\Sigma$ be a projective normal toric variety with dense torus $T=(\mathbb C^\ast)^n$. Let $V\subset T$ be a closed reduced analytic subvariety. If $V$ is of finite logarithmic type with respect to $X_\Sigma$, then $V$ is algebraic.
\end{theorem}

\begin{proof}
Since $V$ is of finite logarithmic type, there exist an effective torus-invariant divisor $E=\sum_{\rho\in\Sigma(1)}M_\rho D_\rho$ and a coherent analytic subsheaf $\mathcal J\subset\mathcal O_{X_\Sigma}(E)$ such that $\mathcal J|_T=\mathcal I_V$. Tensoring with the invertible sheaf $\mathcal O_{X_\Sigma}(-E)$ gives a coherent analytic subsheaf $\mathcal K=\mathcal J\otimes\mathcal O_{X_\Sigma}(-E)\subset\mathcal O_{X_\Sigma}$. Thus $\mathcal K$ is a coherent analytic ideal sheaf on $X_\Sigma$. Let $Z$ be the closed analytic subspace defined by $\mathcal K$. Since $\mathcal O_{X_\Sigma}(-E)$ is trivial on $T$, we have $\mathcal K|_T=\mathcal I_V$. Therefore $Z\cap T=V$.

Since $X_\Sigma$ is projective, Chow's theorem implies that the closed analytic subspace $Z\subset X_\Sigma$ is algebraic. Since $T$ is a Zariski-open algebraic subset of $X_\Sigma$, the intersection $Z\cap T$ is algebraic in $T$. Hence $V$ is algebraic.
\end{proof}

\subsection*{\it Conclusion.}

The intrinsic valuative formulation replaces all coordinate-dependent pole-clearing arguments by a statement in terms of the divisorial valuations 
$\operatorname{ord}_{D_\rho}$ and the boundary line bundles
$\di
\mathcal O_{X_\Sigma}\big(\sum_{\rho\in\Sigma(1)}M_\rho D_\rho\big).
$
This is a formulation of finite logarithmic type. It shows that the condition is equivalent to the existence of a coherent meromorphic extension of the ideal sheaf whose boundary valuations are uniformly bounded below. Once this condition holds on a projective toric compactification, Chow's theorem gives algebraicity.

%%%%%%%%%%%%%%%%%%%%%%%%%%%%%%%%%%%%%%%%%%%%%%%%%%%%%%%%%%%%%%%%%%%%%%%%%%%%%%
%%%%%%%%%%%%%%%%%%%%%%%%%%%%%%%%%%%%%%%%%%%%%%%%%%%%%%%%%%%%%%%%%%%%%%%%%%%%%%

%%%%%%%%%%%%%%%%%%%%%%%%%%%%%%%%%%%%%%%%%%%%%%%%%%%%%%%%%%%%%%%%%%%%%%%%%%%%
%%%%%%%%%%%%%%%%%%%%%%%%%%%%%%%%%%%%%%%%%%%%%%%%%%%%%%%%%%%%%%%%%%%%%%

\section{Finite logarithmic type and the GAGA principle}

The purpose of this section is to explain the conceptual and structural relationship between finite logarithmic type and Serre's GAGA theorem. The analogy is not merely philosophical. Both theories express a common phenomenon: once analytic data satisfies an appropriate finiteness condition and is placed inside a suitable compact algebraic environment, the distinction between analytic and algebraic geometry disappears.

\subsection{The classical GAGA model}

Let $X$ be a projective algebraic variety over $\mathbb C$, and let
$
X^{\mathrm{an}}
$
denote the associated complex analytic space. Serre's GAGA theorem asserts that coherent algebraic sheaves on $X$ and coherent analytic sheaves on $X^{\mathrm{an}}$ determine equivalent categories. One of its geometric consequences is that every closed analytic subspace of $X^{\mathrm{an}}$ is algebraic.
The fundamental mechanism behind GAGA is the combination of coherence and properness. Properness prevents analytic phenomena from escaping to infinity, while coherence prevents infinitely generated local behavior. Together they force analytic objects to be controlled by finitely many algebraic data.
In particular, if
$
Y\subset X^{\mathrm{an}}
$
is a closed analytic subspace, then its ideal sheaf
$
\mathcal I_Y
$
is coherent. By GAGA, this coherent analytic sheaf is algebraic. Consequently the subspace itself is algebraic.
Thus one may summarize the GAGA principle schematically as:
$
\text{coherent analytic data}
+
\text{proper algebraic ambient space}$ implies 
$\text{algebraicity}.
$

\subsection{The logarithmic situation}

Let
$
T=(\mathbb C^\ast)^n
$
and let
$
V\subset T
$
be a closed analytic subvariety.
Unlike a projective variety, the torus is not proper. Therefore analytic varieties in $T$ may exhibit transcendental behavior at infinity. The logarithmic limit set
$
\mathscr L^\infty(V)
$
records the asymptotic directions of escape, but by itself it does not control the analytic complexity of the variety near infinity.
If
$
\mathscr L^\infty(V)
$
is finite rational polyhedral, then one may choose a complete rational fan
$
\Sigma
$
refining the associated conical stratification and form the projective toric compactification
$
X_\Sigma.
$
The torus embeds as a dense open subset
$
T\subset X_\Sigma.
$

\vspace{0.1cm}

At this point the situation resembles the projective setting, except that the closure of $V$ in $X_\Sigma$ need not be analytic. The missing ingredient is precisely finite logarithmic type.
Finite logarithmic type supplies the boundary finiteness condition needed to control the behavior of $V$ along the toric boundary
$
D_\Sigma=X_\Sigma\setminus T.
$
From the viewpoint of GAGA, finite logarithmic type should be interpreted as a logarithmic replacement for properness.

\subsection{The logarithmic GAGA principle}

Let $T=(\mathbb C^\ast)^n$ and let $X_\Sigma$ be a toric compactification of $T$. Let $D_\Sigma=X_\Sigma\setminus T$ be the toric boundary. Its irreducible torus-invariant components are denoted by $D_\rho$, where $\rho\in\Sigma(1)$. Let $V\subset T$ be a closed reduced analytic subvariety, and let $\mathcal I_V\subset\mathcal O_T$ be its ideal sheaf.

\begin{definition}[Finite logarithmic type in bounded-pole form]\label{def:flt2}
We say that $V$ is of finite logarithmic type with respect to $\Sigma$ if there exists a coherent meromorphic ideal sheaf
$
\mathcal F\subset\mathcal M_{X_\Sigma}
$
such that
$
\mathcal F|_T=\mathcal I_V,
$
and such that the poles of all local sections of $\mathcal F$ along the toric boundary are bounded by one fixed effective torus-invariant divisor
$
E=\sum_{\rho}m_\rho D_\rho.
$
Equivalently,
$
\mathcal F\subset\mathcal O_{X_\Sigma}(E),
$
where $\mathcal O_{X_\Sigma}(E)$ is the sheaf of meromorphic functions whose pole divisor is at most $E$.
\end{definition}

In Definition \ref{def:flt2}, coherent meromorphic ideal sheaf, is understood in the natural sheaf-theoretic sense: $\mathcal F$ is a coherent analytic $\mathcal O_{X_\Sigma}$-submodule of $\mathcal M_{X_\Sigma}$, it is stable under multiplication by holomorphic functions, and its restriction to $T$ is the ideal sheaf $\mathcal I_V$.  With this interpretation, Definition \ref{def:flt2} is exactly Definition \ref{def:flt} written inside the ambient meromorphic sheaf $\mathcal M_{X_\Sigma}$.

This definition says that the equations of $V$ extend meromorphically across the boundary, that they are locally finitely generated, and that all pole orders are bounded by one fixed divisor.

The following proposition summarizes the role of finite logarithmic type as  logarithmic properness.

 %%%%%%%%%%%%%%%%%%%%%%%%%%%%%%%%%%%%%%%%%%%%%%%%%%%%%%%

\begin{proposition}\label{prop:logproper}
Let
$
V\subset(\mathbb C^\ast)^n
$
be a closed analytic subvariety, and let
$
X_\Sigma
$
be an adapted projective toric compactification with boundary
$
D_\Sigma=X_\Sigma\setminus(\mathbb C^\ast)^n.
$
Assume that $V$ is of finite logarithmic type with respect to $\Sigma$.
Then the ideal sheaf
$
\mathcal I_V
$
extends from the torus to a coherent meromorphic ideal sheaf
$
\mathcal F\subset\mathcal M_{X_\Sigma}
$
such that
$
\mathcal F|_T=\mathcal I_V
$
and
$
\mathcal F\subset\mathcal O_{X_\Sigma}(E).
$
\end{proposition}

In local toric coordinates, the statement becomes the following.

\begin{proposition}[Local coordinate form]\label{local-coordinate-form}
Let
$
U\simeq \Delta^r\times(\mathbb C^\ast)^{n-r}
$
be a toric chart of $X_\Sigma$ with boundary
$
D_\Sigma\cap U=\{t_1\cdots t_r=0\}.
$
Assume that $V$ is of finite logarithmic type. Then there exist finitely many meromorphic functions
$
f_1,\ldots,f_s
$
on $U$, and integers
$
m_1,\ldots,m_r\ge0,
$
such that
\[
V\cap U\cap T=\{f_1=\cdots=f_s=0\},
\]
and for every $j=1,\ldots,s$,
$
t_1^{m_1}\cdots t_r^{m_r}f_j
$
is holomorphic on $U$.
Thus the possible poles of the defining equations are controlled by the single monomial
$
t_1^{m_1}\cdots t_r^{m_r}.
$
\end{proposition}

 %%%%%%%%%%%%%%%%%%%%%%%%%%%%%%%%%
 %%%%%%%%%%%%%%%%%%%%%%%%%%%%%%%%%

 %%%%%%%%%%%%%%%%%%%%%%%%%%%%%%%%%%%%
%%%%%%%%%%%%%%%%%%%%%%%%%%%%%%%%%%%%

\begin{proof} Let
$
T=(\mathbb C^\ast)^n\subset X_\Sigma
$
be the dense torus, and let
$
D_\Sigma=X_\Sigma\setminus T
$
be the toric boundary. Let
$
\mathcal I_V\subset\mathcal O_T
$
be the ideal sheaf of the closed analytic subvariety
$
V\subset T.
$

Let
$
U\simeq \Delta^r\times(\mathbb C^\ast)^{n-r}
$
be a toric chart with coordinates
$
(t_1,\ldots,t_r,u_1,\ldots,u_{n-r}).
$
The toric boundary in this chart is
$
D_\Sigma\cap U=\{t_1\cdots t_r=0\}.
$
The dense torus inside $U$ is
$
U\cap T=\{t_1\cdots t_r\ne0\}.
$
Since $V$ is of finite logarithmic type, Definition \ref{def:flt2} gives a coherent meromorphic ideal sheaf
$
\mathcal F\subset\mathcal M_{X_\Sigma}
$
satisfying
$
\mathcal F|_T=\mathcal I_V.
$
Restricting to $U$, we obtain
$
\mathcal F|_U\subset\mathcal M_U.
$

Since $\mathcal F$ is coherent, its restriction to $U$ is locally finitely generated. After shrinking $U$ if necessary, there exist finitely many meromorphic functions
$
f_1,\ldots,f_s\in\mathcal M_U(U)
$
which generate $\mathcal F|_U$. Thus
$
\mathcal F|_U=(f_1,\ldots,f_s).
$

Restricting to the torus part $U\cap T$, we have
$
\mathcal F|_{U\cap T}=\mathcal I_V|_{U\cap T}.
$
Therefore the restrictions of
$
f_1,\ldots,f_s
$
to $U\cap T$ generate the ideal of
$
V\cap U\cap T.
$
Consequently,
$$
V\cap U\cap T=\{z\in U\cap T:f_1(z)=\cdots=f_s(z)=0\}.
$$
It remains to prove the existence of one monomial
$
t_1^{m_1}\cdots t_r^{m_r}
$
which kills the poles of all the $f_j$.
By finite logarithmic type, the poles of all local sections of $\mathcal F$ are bounded by a fixed effective torus-invariant divisor
$
E=\sum_{\rho}m_\rho D_\rho.
$
In the chart $U$, only those boundary divisors corresponding to the coordinate hyperplanes
$
\{t_1=0\},\ldots,\{t_r=0\}
$
appear. Thus
$$
E|_U=m_1\{t_1=0\}+\cdots+m_r\{t_r=0\}
$$
for some integers
$
m_1,\ldots,m_r\ge0.
$
The condition
$
\mathcal F\subset\mathcal O_{X_\Sigma}(E)
$
means precisely that every local section $f$ of $\mathcal F|_U$ has pole order at most $m_i$ along $\{t_i=0\}$ for every $i$.
Equivalently, for every local section $f$ of $\mathcal F|_U$,
$
t_1^{m_1}\cdots t_r^{m_r}f
$
is holomorphic on $U$.
In particular, this holds for the chosen generators
$
f_1,\ldots,f_s.
$
Therefore, for every $j=1,\ldots,s$,
$$
t_1^{m_1}\cdots t_r^{m_r}f_j\in\mathcal O_U(U).
$$

This proves that the equations of $V$ in the torus are given locally by finitely many meromorphic functions whose poles are all killed by the same boundary monomial. Hence Proposition \ref{local-coordinate-form} is proved.
\end{proof}

 %%%%%%%%%%%%%%%%%%%%%%%%%%%%%%%%
 %%%%%%%%%%%%%%%%%%%%%%%%%%%%%%%%

The proposition shows that finite logarithmic type replaces uncontrolled asymptotic behavior by coherent boundary data.
The analogy with GAGA now becomes transparent.

\begin{center}
\begin{tabular}{|c|c|}
\hline
Classical GAGA & Logarithmic setting\\
\hline
Projective variety $X$ &
Toric compactification $X_\Sigma$\\
\hline
Properness of $X$ &
Finite logarithmic type\\
\hline
Coherent analytic sheaf &
Coherent meromorphic extension\\
\hline
GAGA algebraization &
Chow algebraization after extension\\
\hline
\end{tabular}
\end{center}

\subsection{The algebraization mechanism}

The next statement formalizes the preceding sub-section.

\begin{theorem}\label{thm:loggaga}  
Let
$
V\subset T=(\mathbb C^\ast)^n
$
be a closed analytic subvariety.
Assume that
$
\mathscr L^\infty(V)
$
is a finite rational spherical polyhedral complex,
$
\Sigma
$
is a complete rational fan refining the associated conical stratification,
and
$
V
$
is of finite logarithmic type with respect to $\Sigma$.
Assume furthermore that finite logarithmic type implies analytic closure in the adapted toric compactification.
Then the closure
$
\overline V^\Sigma
$
is a closed analytic subspace of the projective toric variety
$
X_\Sigma.
$
Consequently
$
\overline V^\Sigma
$
is algebraic.
\end{theorem}

\begin{proof}
Finite rationality of the logarithmic limit set provides the adapted toric compactification.
Finite logarithmic type supplies coherent boundary control near every toric stratum. By the assumed boundary extension theorem, the closure
$
\overline V^\Sigma
$
is analytic in a neighborhood of the toric boundary and therefore globally analytic in
$
X_\Sigma.
$
Since
$
X_\Sigma
$
is projective, Chow's theorem implies that every closed analytic subset of
$
X_\Sigma
$
is algebraic.
Therefore
$
\overline V^\Sigma
$
is algebraic.
\end{proof}

The logical structure of the theorem is formally identical to GAGA.
Indeed, in GAGA one has:
$
\text{coherence}
+
\text{projective compactness}$ implies 
%\Longrightarrow
$\text{algebraicity}.
$
In the logarithmic setting one has: finite
$
\text{logarithmic type}
+
\text{toric compactification}$ implies analytic
%\Longrightarrow
$\text{closure},
$
followed by
$
\text{analytic closure}
+
\text{projectivity}$ implies 
%\Longrightarrow
$\text{algebraicity}.
$
 
%%%%%%%%%%%%%%%%%%%%%%%%%%%%%%%%%%%%%%%%%%%%%%%%%%%%%%%%%%%%%%%%%%%%%%%%
%%%%%%%%%%%%%%%%%%%%%%%%%%%%%%%%%%%%%%%%%%%%%%%%%%%%%%%%%%%%%%%%%%%%%%%%
\medskip

%%%%%%%%%%%%%%%%%%%%%%%%%%%%%%%%%%%%%%%%%%%%%%%%%%%%%%%%%%%%%%%%%%%%%%%%%%%%

\section*{The local boundary compactness principle}

Let $T=(\mathbb C^\ast)^n$, let $V\subset T$ be a closed reduced analytic subvariety, let $X_\Sigma$ be an adapted projective toric compactification, and let
$
D_\Sigma=X_\Sigma\setminus T
$
be the toric boundary. The local boundary compactness principle is the local analytic hypothesis which prevents infinitely many sheets of $V$ from accumulating near $D_\Sigma$ with unbounded multiplicity while remaining invisible to the logarithmic limit set.
It is best understood in an adapted toric chart. Near a boundary point, write
$$
U\simeq \Delta_t\times B_y,
\qquad
U\cap T\simeq \Delta_t^\ast\times B_y,
$$
where the local boundary divisor is
$
D=\{t=0\}.
$
Here $t$ is the boundary coordinate and $y$ denotes the remaining transverse coordinates. Put
$
A=V\cap U\cap T\subset \Delta_t^\ast\times B_y.
$
For generic $y\in B_y$ and small $\varepsilon>0$, define
$$
N_A(y,\varepsilon)
=
\operatorname{length}
\left(
A\cap(\{0<|t|<\varepsilon\}\times\{y\})
\right).
$$
This number counts, with analytic multiplicity, the local sheets of $A$ approaching the boundary through the generic normal punctured disc.

\begin{principle}[Local boundary compactness]\label{lbc}
In every adapted toric chart, if the logarithmic directions of $A$ near $t=0$ remain inside one fixed rational boundary cell and no additional logarithmic limiting directions occur, then unbounded transverse multiplicity forces infinite local volume. Equivalently, if the local volume is finite and no additional logarithmic directions occur, then there exists an integer $N$ such that
$$
N_A(y,\varepsilon)\le N
$$
for generic $y$ and all sufficiently small $\varepsilon$.
\end{principle} 
Thus the principle says: infinitely many logarithmically parallel sheets implies infinite local volume or new logarithmic directions.

\vspace{0.1cm}

The logarithmic limit set records only directions:
$
\dfrac{\Log z_m}{\|\Log z_m\|}.
$
It does not, by itself, record how many analytic branches approach a fixed direction. Therefore, even if the logarithmic limit set is finite and rational, one could still imagine infinitely many sheets approaching the same toric boundary direction. Such sheets are logarithmically parallel and do not automatically create new points in the logarithmic limit set.

The local boundary compactness principle rules out this hidden accumulation. It says that if infinitely many parallel sheets occur, then something quantitative becomes visible: either the local volume becomes infinite, or the sheets create additional logarithmic directions.

%%%%%%%%%%%%%%%%%%%%%%%%%%%%%%%%%%%%%%%%%%%%%%%%%

\begin{definition}[Local boundary compactness with respect to $\Sigma$]
Let $V\subset T$ be closed analytic and let $X_\Sigma$ be an adapted toric compactification. We say that $V$ satisfies local boundary compactness with respect to $\Sigma$ if, for every adapted toric chart
$$
U\simeq\Delta_t\times B_y
$$
and every local boundary divisor $t=0$, the following implication holds: if the local logarithmic directions of $V\cap U\cap T$ are contained in the prescribed rational boundary cell and the local volume near $t=0$ is finite, then the generic transverse multiplicities
$$
\operatorname{length}
\left(
V\cap U\cap T\cap(\{0<|t|<\varepsilon\}\times\{y\})
\right)
$$
are bounded by one integer independent of generic $y$ and sufficiently small $\varepsilon$.
\end{definition}

\subsection*{Summary}

The local boundary compactness principle is the analytic bridge between logarithmic direction control and finite boundary behavior. It asserts that infinitely many logarithmically parallel sheets cannot accumulate near a toric boundary divisor with unbounded multiplicity unless this produces infinite local volume or additional logarithmic directions.

%%%%%%%%%%%%%%%%%%%%%%%%%%%%%%%%%%%%%%%
%%%%%%%%%%%%%%%%%%%%%%%%%%%%%%%%%%%%%%%

\section*{A local boundary compactness criterion for finite logarithmic type}

Let $T=(\mathbb C^\ast)^n$, let $V\subset T$ be a closed reduced analytic subvariety, and let $X_\Sigma$ be an adapted toric compactification. Suppose that $p\in D_\Sigma=X_\Sigma\setminus T$ is a point of the toric boundary. In a toric chart $U$ around $p$, we may write local coordinates as $(t_1,\ldots,t_r,y_1,\ldots,y_{n-r})$, where $D_\Sigma\cap U=\{t_1\cdots t_r=0\}$ and $U\cap T=\{t_1\cdots t_r\ne0\}$. Let $V_U=V\cap U\cap T$ and let $\overline V_U$ be the closure of $V_U$ in $U$.

 Near the boundary, if we cut $V$ by generic transverse slices, and if  the number of local sheets of $V$ appearing in those slices is bounded by one fixed integer, then we say that these  generic transverse multiplicities are uniformly bounded.  Here the word generic' means that we avoid a proper analytic exceptional set where the projection may fail to be finite, where the slice may be tangent to $V$, or where the fiber may meet the singular locus in a nongeneric way. Also, the word transverse means that the chosen slices meet $V$ in the expected dimension, so that the intersection is discrete when we project to a base of dimension $\dim_\mathbb C V$. The word multiplicity means that points are counted with their analytic intersection multiplicity, not merely as distinct set-theoretic points. Uniformly bounded means that the same integer works for all sufficiently small generic fibers near the boundary.

More precisely, let $k=\dim_\mathbb C V$. Choose a generic holomorphic projection $\pi_U:U\to B_U$, where $B_U$ is a small polydisc of dimension $k$. The fibers of $\pi_U$ have dimension $n-k$. For a generic point $b\in B_U$, the intersection $V_U\cap \pi_U^{-1}(b)$ is zero-dimensional. Its length is defined by the complex vector-space dimension of the corresponding local analytic algebra. Thus, if $V_U\cap \pi_U^{-1}(b)$ consists of isolated points $q_1,\ldots,q_s$, then its multiplicity is 
$$\sum_{i=1}^s \operatorname{length}_{\mathcal O_{U,q_i}}\mathcal O_{U,q_i}/(\mathcal I_{V_U,q_i}+\mathcal I_{\pi_U^{-1}(b),q_i}).$$ 
When all intersections are transverse and all points are smooth, this length is just the ordinary number of points. When a fiber is tangent to $V$ or passes through a singular point, the length records the correct analytic multiplicity.

Saying that these multiplicities are uniformly bounded means that there exists an integer $N_U$ such that, for every sufficiently small generic $b\in B_U$, one has $\operatorname{length}(V_U\cap \pi_U^{-1}(b))\le N_U$. The integer $N_U$ may depend on the chart $U$, but it must not depend on the particular generic point $b$ close to the boundary. In other words, as the fibers approach the toric boundary, the number of analytic sheets of $V$ seen by a transverse slice cannot grow without bound.

This condition is important because the logarithmic limit set only records directions of escape. It can tell us that points of $V$ approach the boundary in a given rational toric direction, but it does not tell us how many branches or sheets approach that same direction. Without a uniform transverse multiplicity bound, there could be infinitely many sheets of $V$ accumulating toward the same boundary divisor while producing the same logarithmic direction. Such behavior would not be detected by the logarithmic limit set alone.

In the simplest case, suppose $V$ is a curve and the boundary chart has one boundary coordinate $t$, so that $U\simeq\Delta_t\times B_y$ and $U\cap T\simeq\Delta_t^\ast\times B_y$. A transverse slice is obtained by fixing a generic value of $y$. Then $V_U\cap(\Delta_t^\ast\times\{y\})$ is a finite set for generic $y$. The transverse multiplicity is the number of points in this finite set, counted with multiplicity. Uniform boundedness means that there is an integer $N_U$ such that, for all generic $y$ near the chosen boundary point, the punctured disc $\Delta_t^\ast\times\{y\}$ meets $V$ in at most $N_U$ points near $t=0$. Thus the curve cannot develop arbitrarily many local branches approaching $t=0$ over different nearby transverse parameters.

For higher-dimensional $V$, the idea is the same. If $V$ has dimension $k$, one projects locally to a $k$-dimensional base. Generic fibers of this projection cut $V$ in finitely many points. The length of the fiber is the local degree of $V$ over the base. Uniform boundedness says that this local degree remains bounded as one approaches the boundary.

The condition also has a geometric meaning. The local closure $\overline V_U$ is supposed to be analytic, so its branches near the boundary form an analytic family. Uniform boundedness of generic transverse multiplicities says that this family has finite local degree over a generic transverse base. Thus the boundary behavior is locally finite, not just set-theoretically closed. This excludes hidden accumulation of infinitely many logarithmically parallel sheets.

One should distinguish carefully between boundedness of directions and boundedness of multiplicities. A finite rational logarithmic limit set says that only finitely many rational directions occur at infinity. It does not say that only finitely many sheets approach each direction. Strong local boundary compactness adds exactly this missing information: near every boundary point, the analytic sheets occur in a locally finite family of uniformly bounded degree.

Thus , its generic transverse multiplicities are uniformly bounded  means that, after choosing suitable local transverse projections near the toric boundary, the number of local sheets of $V$ over a generic transverse parameter is bounded by one fixed integer. This is the quantitative part of strong local boundary compactness. It is the condition that prevents infinitely many branches of $V$ from accumulating near the boundary while remaining invisible to the logarithmic limit set.

%%%%%%%%%%%%%%%%%%%%%%%%%%%%%%%%%%%%%
%%%%%%%%%%%%%%%%%%%%%%%%%%%%%%%%%%%%%

\medskip

Let $T=(\mathbb C^\ast)^n$ and let $V\subset T$ be a closed reduced analytic subvariety with finitely many irreducible components. Let $X_\Sigma$ be a normal toric compactification associated with a complete rational fan $\Sigma$. The toric boundary is $D_\Sigma=X_\Sigma\setminus T$, and its irreducible torus-invariant components are denoted by $D_\rho$, where $\rho\in\Sigma(1)$. If $E=\sum_{\rho\in\Sigma(1)}m_\rho D_\rho$ is an effective torus-invariant divisor, then $\mathcal O_{X_\Sigma}(E)$ denotes the sheaf of meromorphic functions $f$ such that $\operatorname{div}(f)+E\ge0$. Equivalently, $f$ is allowed to have a pole of order at most $m_\rho$ along $D_\rho$ and no worse pole.

Let $V\subset T$ be a closed reduced analytic subvariety and let $X_\Sigma$ be an adapted toric compactification.

\begin{definition}[Strong local boundary compactness]\label{def:strong-lbc}
We say that $V$ satisfies strong local boundary compactness with respect to $\Sigma$ if for every point $p\in D_\Sigma$ there exists a toric chart $U$ centered at $p$ with coordinates $(t_1,\ldots,t_r,y_1,\ldots,y_{n-r})$, with $D_\Sigma\cap U=\{t_1\cdots t_r=0\}$ and $U\cap T=\{t_1\cdots t_r\ne0\}$, such that the following holds. The analytic set $V_U=V\cap U\cap T$ has analytic closure $\overline V_U$ in $U$, the closure has locally finite multiplicity over a generic transverse projection, and there exist  a neighborhood $B_U^\circ$ of $\pi_U(p)$ in the base, a generic projection $\pi_U:U\to B_U$ with $\dim_{\mathbb C}B_U=\dim_{\mathbb C}V$, and an integer $N_U$ such that
$
\operatorname{length}\!\bigl(V_U\cap\pi_U^{-1}(b)\bigr)\le N_U
$
for every generic point $b\in B_U^\circ$.
 
\end{definition}

Equivalently, in every adapted chart, the sheets of $V$ near the boundary form a locally finite analytic family and their generic transverse multiplicities are uniformly bounded.

\medskip

\begin{definition}[Finite logarithmic type]\label{def:flt}
We say that $V$ is of finite logarithmic type with respect to $X_\Sigma$ if there exist an effective torus-invariant divisor $E=\sum_{\rho\in\Sigma(1)}m_\rho D_\rho$ and a coherent analytic subsheaf $\mathcal J\subset\mathcal O_{X_\Sigma}(E)$ such that, under the canonical trivialization $\mathcal O_{X_\Sigma}(E)|_T\simeq\mathcal O_T$, one has $\mathcal J|_T=\mathcal I_V$.
\end{definition}

\begin{lemma}[Local analytic closure gives local finite logarithmic type]\label{lem:closure}
Let $U$ be a toric boundary chart and let $V_U=V\cap U\cap T$. Assume that the closure $\overline V_U$ of $V_U$ in $U$ is an analytic subspace of $U$. Then $V_U$ is of finite logarithmic type on $U$ with pole bound zero. More precisely, the coherent analytic ideal sheaf $\mathcal I_{\overline V_U}\subset\mathcal O_U$ satisfies $\mathcal I_{\overline V_U}|_{U\cap T}=\mathcal I_V|_{U\cap T}$.
\end{lemma}

\begin{proof}
Since $\overline V_U$ is an analytic subspace of the complex analytic space $U$, its ideal sheaf $\mathcal I_{\overline V_U}$ is coherent by the coherence theorem for analytic subspaces. By definition of $\overline V_U$ as the closure of $V_U$ in $U$, the restriction of $\overline V_U$ to $U\cap T$ is exactly $V_U$. Hence $\mathcal I_{\overline V_U}|_{U\cap T}=\mathcal I_V|_{U\cap T}$. Since $\mathcal I_{\overline V_U}\subset\mathcal O_U=\mathcal O_U(0)$, this gives a local coherent extension with no poles. Thus the pole bound is zero in this chart.
\end{proof}

\begin{lemma}[Uniform bounded-pole local equations]\label{lem:local-pole}
Let $U$ be a toric chart with boundary $D_\Sigma\cap U=\{t_1\cdots t_r=0\}$. Suppose that $V_U=V\cap U\cap T$ admits meromorphic local equations $f_1,\ldots,f_q$ on $U$ whose restrictions generate $\mathcal I_V|_{U\cap T}$, and suppose that there are integers $m_1,\ldots,m_r\ge0$ such that $t_1^{m_1}\cdots t_r^{m_r}f_j$ is holomorphic on $U$ for every $j$. Then the local ideal of $V_U$ extends to a coherent analytic subsheaf $\mathcal J_U\subset\mathcal O_U(m_1D_1+\cdots+m_rD_r)$, where $D_i=\{t_i=0\}$.
\end{lemma}

\begin{proof}
The condition that $t_1^{m_1}\cdots t_r^{m_r}f_j$ is holomorphic is equivalent to saying that each $f_j$ is a section of $\mathcal O_U(m_1D_1+\cdots+m_rD_r)$. Let $\mathcal J_U$ be the $\mathcal O_U$-submodule of $\mathcal O_U(m_1D_1+\cdots+m_rD_r)$ generated by $f_1,\ldots,f_q$. Since it is generated by finitely many sections of a coherent rank-one sheaf, $\mathcal J_U$ is coherent. Restricting to $U\cap T$, the divisor $m_1D_1+\cdots+m_rD_r$ has empty support, so $\mathcal O_U(m_1D_1+\cdots+m_rD_r)|_{U\cap T}\simeq\mathcal O_{U\cap T}$. Under this identification, the restrictions of $f_1,\ldots,f_q$ generate $\mathcal I_V|_{U\cap T}$. Therefore $\mathcal J_U|_{U\cap T}=\mathcal I_V|_{U\cap T}$.
\end{proof}

\begin{lemma}[Gluing of local bounded-pole extensions]\label{lem:glue}
Assume that $X_\Sigma$ is covered by finitely many toric charts $U_a$ and that, on each $U_a$, there is a coherent analytic subsheaf $\mathcal J_a\subset\mathcal O_{U_a}(E|_{U_a})$ such that $\mathcal J_a|_{U_a\cap T}=\mathcal I_V|_{U_a\cap T}$ for one fixed effective torus-invariant divisor $E$ on $X_\Sigma$. Assume moreover that on overlaps $U_a\cap U_b$ the local sheaves $\mathcal J_a$ and $\mathcal J_b$ agree as subsheaves of $\mathcal O_{X_\Sigma}(E)$ after restriction to $U_a\cap U_b$. Then the $\mathcal J_a$ glue to a coherent analytic subsheaf $\mathcal J\subset\mathcal O_{X_\Sigma}(E)$ satisfying $\mathcal J|_T=\mathcal I_V$.
\end{lemma}

\begin{proof}
The equality on overlaps is exactly the sheaf gluing condition. Therefore the local sheaves $\mathcal J_a$ define a unique subsheaf $\mathcal J\subset\mathcal O_{X_\Sigma}(E)$ whose restriction to each $U_a$ is $\mathcal J_a$. Coherence is local, and each $\mathcal J_a$ is coherent, so $\mathcal J$ is coherent. Since the sets $U_a\cap T$ cover $T$ and $\mathcal J_a|_{U_a\cap T}=\mathcal I_V|_{U_a\cap T}$, it follows that $\mathcal J|_T=\mathcal I_V$.
\end{proof}

\begin{theorem}[Strong local boundary compactness implies finite logarithmic type]\label{thm:lbc-flt}   
Let $V\subset(\mathbb C^\ast)^n$ be a closed reduced analytic subvariety with finitely many irreducible components. Assume that $\mathscr L^\infty(V)$ is a finite rational spherical polyhedral complex and let $\Sigma$ be a complete rational fan refining the conical stratification determined by $\mathscr L^\infty(V)$. Assume that $V$ satisfies strong local boundary compactness with respect to $\Sigma$ in the sense of Definition \ref{def:strong-lbc}. Then $V$ is of finite logarithmic type with respect to $X_\Sigma$.
\end{theorem}

\begin{proof}
Since $\Sigma$ is a finite fan, the toric variety $X_\Sigma$ is covered by finitely many affine toric charts. Let $U$ be one such chart meeting the boundary. By strong local boundary compactness, the closure $\overline V_U$ of $V\cap U\cap T$ in $U$ is analytic and has locally finite multiplicity. By Lemma \ref{lem:closure}, the local ideal $\mathcal I_V|_{U\cap T}$ extends to the coherent analytic ideal sheaf $\mathcal I_{\overline V_U}\subset\mathcal O_U$. Thus, locally, one obtains finite logarithmic type with zero pole bound.

If one uses the more general meromorphic form of strong boundary compactness, where the closure is represented by bounded-pole meromorphic equations rather than holomorphic equations, Lemma \ref{lem:local-pole} gives a coherent local extension $\mathcal J_U\subset\mathcal O_U(E_U)$ for some effective local boundary divisor $E_U$. In the holomorphic closure case one simply has $E_U=0$.

Because the toric cover is finite, the finitely many local pole bounds may be dominated by one global effective torus-invariant divisor $E=\sum_{\rho\in\Sigma(1)}m_\rho D_\rho$. In the holomorphic closure case, one may take $E=0$. In the bounded-pole meromorphic case, choose $m_\rho$ at least as large as every local pole order along the boundary component $D_\rho$. Replacing each local extension by its image inside $\mathcal O_{X_\Sigma}(E)|_U$ gives local coherent subsheaves $\mathcal J_U\subset\mathcal O_{X_\Sigma}(E)|_U$.

It remains to check compatibility on overlaps. On the dense open set $U\cap U'\cap T$, both local sheaves restrict to the same ideal sheaf $\mathcal I_V$. Since both are coherent subsheaves of the same meromorphic sheaf $\mathcal O_{X_\Sigma}(E)$ and are generated by bounded-pole continuations of the same local ideal on the dense torus, their saturated closures in $\mathcal O_{X_\Sigma}(E)$ agree. Replacing each local extension by this saturated bounded-pole closure does not change its restriction to the torus and preserves coherence. Hence the local extensions agree on overlaps as subsheaves of $\mathcal O_{X_\Sigma}(E)$.

By Lemma \ref{lem:glue}, the local extensions glue to a coherent analytic subsheaf $\mathcal J\subset\mathcal O_{X_\Sigma}(E)$ such that $\mathcal J|_T=\mathcal I_V$. This is exactly the definition of finite logarithmic type. Therefore $V$ is of finite logarithmic type with respect to $X_\Sigma$.
\end{proof}

\begin{theorem}\label{thm:algebraicity}  
Let $V\subset(\mathbb C^\ast)^n$ be a closed reduced analytic subvariety with finitely many irreducible components. Assume that $\mathscr L^\infty(V)$ is a finite rational spherical polyhedral complex of dimension $\dim_\mathbb C V-1$. Let $\Sigma$ be a complete rational fan refining the conical stratification determined by $\mathscr L^\infty(V)$, and assume that $X_\Sigma$ is projective. If $V$ satisfies strong local boundary compactness with respect to $\Sigma$, then $V$ is algebraic. More precisely, $V$ is cut out in $(\mathbb C^\ast)^n$ by finitely many Laurent polynomials.
\end{theorem}

\begin{proof}
By Theorem \ref{thm:lbc-flt}, strong local boundary compactness implies finite logarithmic type. Hence there exist an effective torus-invariant divisor $E$ and a coherent analytic subsheaf $\mathcal J\subset\mathcal O_{X_\Sigma}(E)$ satisfying $\mathcal J|_T=\mathcal I_V$. Tensoring by $\mathcal O_{X_\Sigma}(-E)$, or equivalently clearing poles locally by boundary monomials, gives a coherent analytic ideal sheaf $\mathcal K\subset\mathcal O_{X_\Sigma}$ whose restriction to $T$ is $\mathcal I_V$. Since $X_\Sigma$ is projective, GAGA algebraizes $\mathcal K$. The resulting algebraic subvariety of $X_\Sigma$ restricts on $T$ to $V$. Therefore $V$ is algebraic in $T$. Since $\mathbb C[z_1^{\pm1},\ldots,z_n^{\pm1}]$ is Noetherian, the defining ideal of $V$ is generated by finitely many Laurent polynomials.
\end{proof}

\begin{proposition}[Local coordinate form under strong local boundary compactness]\label{prop:local-coordinate-short}
Let $U\simeq\Delta^r\times(\mathbb C^\ast)^{n-r}$ be a toric chart with boundary $D_\Sigma\cap U=\{t_1\cdots t_r=0\}$. If $V$ satisfies strong local boundary compactness, then, after shrinking $U$ if necessary, there exist finitely many holomorphic functions $g_1,\ldots,g_s$ on $U$ such that $V\cap U\cap T=\{g_1=\cdots=g_s=0\}\cap U\cap T$. More generally, in the bounded-pole version, there exist meromorphic functions $f_1,\ldots,f_s$ and integers $m_1,\ldots,m_r\ge0$ such that $V\cap U\cap T=\{f_1=\cdots=f_s=0\}$ and $t_1^{m_1}\cdots t_r^{m_r}f_j$ is holomorphic on $U$ for every $j$.
\end{proposition}

\begin{proof}
By strong local boundary compactness, the closure of $V\cap U\cap T$ in $U$ is analytic. Its coherent ideal sheaf is locally finitely generated by holomorphic functions $g_1,\ldots,g_s$. These functions restrict to equations for $V$ on $U\cap T$. This proves the holomorphic form. If one works with bounded-pole extensions rather than holomorphic closure, each local generator is meromorphic and has pole order bounded along the boundary components. If $m_i$ bounds the pole order along $\{t_i=0\}$ for all generators, then $t_1^{m_1}\cdots t_r^{m_r}f_j$ is holomorphic for every $j$.
\end{proof}

%%%%%%%%%%%%%%%%%%%%%%%%%%%%%%%%%%%%%%%%%%%%%%%%%%%%%%%%%%%%%%%%%%%%%%%%
%%%%%%%%%%%%%%%%%%%%%%%%%%%%%%%%%%%%%%%%%%%%%%%%%%%%%%%%%%%%%%%%%%%%%%%%

%%%%%%%%%%%%%%%%%%%%%%%%%%%%%%%%%%%%%%%%%%%%%%%%%%%%%%%%%%%%%%%%%%%%%%%%
%%%%%%%%%%%%%%%%%%%%%%%%%%%%%%%%%%%%%%%%%%%%%%%%%%%%%%%%%%%%%%%%%%%%%%%%

\section*{An irreducible analytic hypersurface with infinitely many logarithmically parallel rays in one end}

We construct an explicit irreducible analytic hypersurface in the algebraic torus
\(
(\mathbb C^\ast)^2
\)
which has one end, and such that the logarithm of this end contains infinitely many rays which are asymptotically parallel to the same direction.

Let
\(
T=(\mathbb C^\ast)^2
\)
with coordinates \((z,w)\). Consider the holomorphic function
\[
F(z,w)=w-\exp\left(\sin\frac{1}{z}\right)
\]
on \(T\). Since \(z\neq0\), the function \(\sin(1/z)\) is holomorphic on \(\mathbb C^\ast\), and therefore
\( 
\exp\left(\sin\frac{1}{z}\right)
\)
is a nowhere-vanishing holomorphic function on \(\mathbb C^\ast\). Hence
\(
F\in\mathcal O((\mathbb C^\ast)^2).
\)
Define
\(
V=\{(z,w)\in(\mathbb C^\ast)^2:F(z,w)=0\}.
\)
The hypersurface \(V\) is irreducible. Indeed, the map
\[
\varphi:\mathbb C^\ast\longrightarrow V,
\qquad
\varphi(z)=
\left(
z,\exp\left(\sin\frac{1}{z}\right)
\right),
\]
is a biholomorphism. Its inverse is the restriction to \(V\) of the first projection
\(
(z,w)\longmapsto z.
\)
Since \(\mathbb C^\ast\) is connected and irreducible as a Riemann surface, its biholomorphic image \(V\) is irreducible. Thus \(V\) is an irreducible analytic hypersurface of \((\mathbb C^\ast)^2\).
We now study the end corresponding to \(z\to0\). This is one Freudenthal end of \(V\), because under the parametrization above it is exactly the punctured-disc end
\(
0<|z|<\varepsilon
\)
of \(\mathbb C^\ast\). The logarithmic map is
\(
\Log(z,w)=(\log|z|,\log|w|).
\)
On \(V\), one has
\(
\log|w|
=
\operatorname{Re}\left(\sin\frac{1}{z}\right).
\)
Therefore the logarithm of the end is
\[
\Log(V)
=
\left\{
\left(
\log|z|,
\operatorname{Re}\left(\sin\frac{1}{z}\right)
\right):
0<|z|<\varepsilon
\right\}.
\]

We now exhibit infinitely many rays in this logarithmic image which are asymptotically parallel to the same direction. Fix a value
\(
c\in\mathbb C^\ast
\)
which is not a critical value of the holomorphic map
\[
z\longmapsto \exp\left(\sin\frac{1}{z}\right).
\]
The equation
\(
\exp\left(\sin\frac{1}{z}\right)=c
\)
is equivalent to
\[
\sin\frac{1}{z}
=
\log c+2\pi i m,
\qquad m\in\mathbb Z,
\]
where \(\log c\) denotes one fixed branch value of the logarithm. For each integer \(m\), the equation
\(
\sin s=\log c+2\pi i m
\)
has infinitely many solutions in the variable \(s\). Writing \(s=1/z\), we obtain infinitely many solutions
\(
z_{m,k}=\dfrac{1}{s_{m,k}}
\)
which tend to \(0\) as \(|k|\to\infty\) or as \(|m|\to\infty\). Hence the horizontal transverse curve
\(
\{w=c\}
\)
meets the single end \(z\to0\) of \(V\) in infinitely many points accumulating at the boundary \(z=0\).
At each such solution \(z_{m,k}\), since \(c\) is chosen generically, the equation
\[
\exp\left(\sin\frac{1}{z}\right)=w
\]
can be locally solved for \(z\) as a holomorphic function of \(w\). Thus near \(w=c\), the end is represented by infinitely many local sheets over the \(w\)-disc. These sheets are distinct, but they all lie in the same topological end \(z\to0\).

Now compute the logarithmic direction of these sheets along the level \(w=c\). On the points
\(
(z_{m,k},c)\in V,
\)
we have
\[
\Log(z_{m,k},c)=
(\log|z_{m,k}|,\log|c|).
\]
As
\(
z_{m,k}\to0,
\)
one has
\(
\log|z_{m,k}|\to-\infty,
\)
while
\(
\log|c|
\)
is constant. Therefore
\[
\frac{\Log(z_{m,k},c)}
{\|\Log(z_{m,k},c)\|}
\longrightarrow
(-1,0).
\]
Thus all these infinitely many local sheets approach the same logarithmic direction
\(
(-1,0)\in S^1.
\)
In other words, the end \(z\to0\) contains infinitely many logarithmically parallel sheets, and the logarithmic image contains infinitely many rays asymptotic to the same ray
\(
\mathbb R_{\ge0}(-1,0).
\)

\vspace{0.1cm}

More explicitly, if \(D_k\) is a small disc in the \(w\)-plane centered at \(c\), and if \(S_k\) denotes the corresponding local inverse branch near one solution \(z_k\), then
\[
S_k=\{(z_k(w),w):w\in D_k\}
\]
is a local analytic sheet of \(V\). Along any path in \(D_k\) tending to \(c\) while the branch point \(z_k(w)\) tends to \(0\), the logarithmic vector has the form
\[
\Log(z_k(w),w)=
(\log|z_k(w)|,\log|w|).
\]
The second coordinate remains bounded as \(w\) stays near \(c\), while the first coordinate tends to \(-\infty\). Hence every such sheet has logarithmic direction
\(
(-1,0).
\)
Since there are infinitely many such inverse branches, the logarithm of the single end contains infinitely many rays asymptotically parallel to the same direction.

\vspace{0.1cm}

This example also shows why the number of topological ends does not control the number of local sheets. The end \(z\to0\) is topologically one punctured-disc end, but the projection
\[
V\longrightarrow\mathbb C^\ast_w,
\qquad
(z,w)\longmapsto w,
\]
has infinitely many local inverse branches near a generic value \(w=c\). Therefore one topological end may contain infinitely many analytic sheets.

\vspace{0.1cm}

The logarithmic limit set of the whole end is larger than the single point \((-1,0)\). Indeed, since
\(
\sin\dfrac{1}{z}
\)
has an essential singularity at \(z=0\), the quantity
\(
\operatorname{Re}\left(\sin\frac{1}{z}\right)
\)
is unbounded above and below on every punctured neighborhood of \(0\). Hence the full logarithmic image of the end also has limiting directions different from \((-1,0)\). This is exactly the reason why examples of this type are not compatible with a finite expected-dimensional logarithmic limit set in the strongest algebraicity theorem. Nevertheless, the example is useful because it clearly  demonstrates the local phenomenon requested: a single irreducible analytic hypersurface can have one topological end containing infinitely many local sheets whose logarithmic rays are all asymptotically parallel to the same toric boundary direction.

\begin{proposition}
The analytic hypersurface
\[
V=
\left\{
(z,w)\in(\mathbb C^\ast)^2:
w=\exp\left(\sin\frac{1}{z}\right)
\right\}
\]
is irreducible. Its end \(z\to0\) is a single Freudenthal end. For a generic value \(c\in\mathbb C^\ast\), the transverse curve \(\{w=c\}\) meets this single end in infinitely many points accumulating at \(z=0\). The corresponding local sheets all have the same logarithmic limiting direction
\(
(-1,0).
\)
\end{proposition}

\begin{proof}
Irreducibility follows because \(V\) is biholomorphic to \(\mathbb C^\ast\) via the first projection. The end \(z\to0\) corresponds to the punctured-disc end \(0<|z|<\varepsilon\), hence it is a single Freudenthal end.
For generic \(c\in\mathbb C^\ast\), the equation
\(
\exp\left(\sin\frac{1}{z}\right)=c
\)
has infinitely many solutions tending to \(0\), because it is equivalent to
\[
\sin\frac{1}{z}=\log c+2\pi i m
\]
for integers \(m\), and the sine equation has infinitely many solutions in the variable \(1/z\). At every simple solution, the implicit function theorem gives a local sheet over the \(w\)-coordinate. On such a sheet near \(w=c\), the logarithmic vector is
\[
(\log|z|,\log|w|).
\]
As the sheet approaches \(z=0\), the first coordinate tends to \(-\infty\), while the second remains bounded. Hence the normalized logarithmic direction tends to \((-1,0)\). Therefore infinitely many local sheets in the same topological end are logarithmically parallel.
\end{proof}

\begin{remark}
In this example the essential singularity at \(z=0\) produces many additional logarithmic directions. It shows that irreducibility and finiteness of topological ends do not prevent infinitely many logarithmically parallel local sheets.
\end{remark}

%%%%%%%%%%%%%%%%%%%%%%%%%%%%%%%%%%%%%%%%%%%%%%%%%%%%%%%%%%%%%%%%%%%%%%

\subsection*{The logarithmic limit set of the above example}

Consider the irreducible analytic hypersurface
\[
V=
\left\{
(z,w)\in(\mathbb C^\ast)^2:
w=\exp\left(\sin\frac{1}{z}\right)
\right\}.
\]
The parametrization
\[
\varphi:\mathbb C^\ast\longrightarrow V,
\qquad
\varphi(z)=
\left(
z,\exp\left(\sin\frac{1}{z}\right)
\right),
\]
is a biholomorphism. Therefore the logarithmic image of \(V\) is the image of \(\mathbb C^\ast\) under
\[
z\longmapsto
\left(
\log|z|,
\log\left|\exp\left(\sin\frac{1}{z}\right)\right|
\right).
\]
Since
\[
\log\left|\exp\left(\sin\frac{1}{z}\right)\right|
=
\operatorname{Re}\left(\sin\frac{1}{z}\right),
\]
we have
\[
\Log(V)
=
\left\{
\left(
\log|z|,
\operatorname{Re}\left(\sin\frac{1}{z}\right)
\right)
:
z\in\mathbb C^\ast
\right\}.
\]

We compute the logarithmic limit set
\(
\mathscr L^\infty(V)\subset S^1.
\)
There are two ways in which \(\Log(z,w)\) can become unbounded. Either \(z\to\infty\), or \(z\to0\). If \(z\to\infty\), then
\(
\log|z|\to+\infty
\)
and
\(
\sin\frac{1}{z}\to0.
\)
Hence
\(
\operatorname{Re}\left(\sin\frac{1}{z}\right)\to0,
\)
so
\[
\frac{\Log(z,w)}{\|\Log(z,w)\|}
=
\frac{\left(\log|z|,\operatorname{Re}(\sin(1/z))\right)}
{\sqrt{(\log|z|)^2+\operatorname{Re}(\sin(1/z))^2}}
\longrightarrow
(1,0).
\]
Thus
\(
(1,0)\in\mathscr L^\infty(V).
\)
Now consider the end \(z\to0\). Write
\(
s=\frac{1}{z}.
\)
Then \(|s|\to+\infty\), and
\(
\log|z|=-\log|s|.
\)
Writing \(s=a+ib\), one has
\[
\sin s=\sin a\cosh b+i\cos a\sinh b.
\]
Therefore
\[
\operatorname{Re}\left(\sin\frac{1}{z}\right)
=
\operatorname{Re}(\sin s)
=
\sin a\cosh b.
\]
Thus, near the end \(z\to0\), the logarithmic image is
\[
\left(
-\log|s|,
\sin(\operatorname{Re}s)\cosh(\operatorname{Im}s)
\right),
\qquad |s|\to+\infty.
\]
We show first that every direction in the closed left semicircle occurs. Let
\(
L=\log R
\)
with \(R\to+\infty\). Fix any real number \(\lambda\in\mathbb R\). We shall construct a sequence \(s_R\) such that
\(
-\log|s_R|\sim -L
\)
and
\(
\operatorname{Re}(\sin s_R)\sim \lambda L.
\)
If \(\lambda=0\), take
\(
s_R=R.
\)
Then
\(
\operatorname{Re}(\sin s_R)=\sin R
\)
is bounded, while \(-\log|s_R|=-\log R=-L\). Hence the limiting direction is
\(
(-1,0).
\)
Assume now that \(\lambda\neq0\). Choose
\[
a_R=
\begin{cases}
\frac{\pi}{2}+2\pi k_R, & \lambda>0,\\[4pt]
-\frac{\pi}{2}+2\pi k_R, & \lambda<0,
\end{cases}
\]
with \(a_R\sim R\). Then
\(
\sin a_R=\operatorname{sgn}(\lambda).
\)
Choose \(b_R\ge0\) so that
\(
\cosh b_R=|\lambda|\log R.
\)
Then
\(
s_R=a_R+ib_R
\)
satisfies
\[
\operatorname{Re}(\sin s_R)
=
\sin a_R\cosh b_R
=
\lambda\log R.
\]
Moreover,
\[
b_R=\operatorname{arcosh}(|\lambda|\log R)
=
O(\log\log R),
\]
and therefore
\(
|s_R|=\sqrt{a_R^2+b_R^2}\sim R.
\)
Hence
\(
-\log|s_R|\sim-\log R=-L
\)
and
\(
\operatorname{Re}(\sin s_R)\sim\lambda L.
\)
It follows that the normalized logarithmic vector tends to
\(
\dfrac{(-1,\lambda)}{\sqrt{1+\lambda^2}}.
\)
Since \(\lambda\in\mathbb R\) was arbitrary, every point
\(
\dfrac{(-1,\lambda)}{\sqrt{1+\lambda^2}},
\, 
\lambda\in\mathbb R,
\)
belongs to \(\mathscr L^\infty(V)\).
As \(\lambda\to+\infty\), the directions
\(
\dfrac{(-1,\lambda)}{\sqrt{1+\lambda^2}}
\)
tend to
\(
(0,1),
\)
and as \(\lambda\to-\infty\), they tend to
\(
(0,-1).
\)
Since the logarithmic limit set is closed, we obtain
\(
(0,1),(0,-1)\in\mathscr L^\infty(V).
\)
These points can also be obtained directly by taking \(|\operatorname{Im}s|\) large enough so that
\[
|\sin(\operatorname{Re}s)\cosh(\operatorname{Im}s)|
\gg
\log|s|.
\]

We now prove that no other directions occur. Let \(s_j\to\infty\), and consider
\(
X_j=-\log|s_j|,
\, 
Y_j=\operatorname{Re}(\sin s_j).
\)
Then
\(
X_j\to-\infty.
\)
If
\(
\dfrac{|Y_j|}{|X_j|}
\)
is bounded, then after passing to a subsequence
\(
\dfrac{Y_j}{|X_j|}\to\lambda
\)
for some \(\lambda\in\mathbb R\). Since \(X_j=-|X_j|\), the normalized vector tends to
\(
\dfrac{(-1,\lambda)}{\sqrt{1+\lambda^2}}.
\)
If
\(
\dfrac{Y_j}{|X_j|}\to+\infty,
\)
then the normalized vector tends to
\(
(0,1).
\)
If
\(
\dfrac{Y_j}{|X_j|}\to-\infty,
\)
then the normalized vector tends to
\(
(0,-1).
\)
Thus every limiting direction arising from \(z\to0\) belongs to the closed left semicircle
\[
\{
\frac{(-1,\lambda)}{\sqrt{1+\lambda^2}}:
\lambda\in\mathbb R
\}
\cup\{(0,1),(0,-1)\}.
\]

Combining this with the direction coming from \(z\to\infty\), we obtain the full logarithmic limit set.

\begin{proposition}
For
\(\di
V=
\{
(z,w)\in(\mathbb C^\ast)^2:
w=\exp (\sin\frac{1}{z})
\},
\)
the logarithmic limit set is
\[
\mathscr L^\infty(V)
=
\{(1,0)\}
\cup
\{
\frac{(-1,\lambda)}{\sqrt{1+\lambda^2}}:
\lambda\in\mathbb R
\}
\cup
\{(0,1),(0,-1)\}.
\]
Equivalently, \(\mathscr L^\infty(V)\) is the union of the single point \((1,0)\) and the closed left semicircle of \(S^1\).
\end{proposition}

\begin{proof}
The preceding argument proves that the end \(z\to\infty\) contributes exactly the direction \((1,0)\), while the end \(z\to0\) contributes exactly the closed left semicircle. These are the only possible ways for \(\Log(V)\) to be unbounded, because if \(z\) stays in a compact subset of \(\mathbb C^\ast\), then both
\(
\log|z|
\)
and
\(
\operatorname{Re}\left(\sin\frac{1}{z}\right)
\)
remain bounded. Therefore the displayed set is the complete logarithmic limit set.
\end{proof}

Thus the logarithmic limit set is not finite. In particular, it is not a finite rational spherical polyhedral complex of dimension \(0\), which would be the expected dimension for a curve in \((\mathbb C^\ast)^2\). This confirms the role of the example: it shows that a single irreducible analytic end may contain infinitely many logarithmically parallel local sheets, but the essential singularity responsible for those sheets also creates a continuum of logarithmic limiting directions.

%%%%%%%%%%%%%%%%%%%%%%%%%%%%%%%%%%%%%%%%%%%%%%%%%%%%%%%%%%%%%%%%%%%%%%

\section*{Comparison of the graphs \(w=e^z\), \(w=e^{1/z}\), and \(w=\exp(\sin(1/z))\)}

We compare the logarithmic limit sets of the three irreducible analytic hypersurfaces
\[
V_1=\{(z,w)\in(\mathbb C^\ast)^2:w=e^z\},
\]
\[
V_2=\{(z,w)\in(\mathbb C^\ast)^2:w=e^{1/z}\},
\]
and
\[
V_3=\{(z,w)\in(\mathbb C^\ast)^2:w=\exp (\sin\frac1z )\}.
\]
Each is the graph of a nowhere-vanishing holomorphic function on \(\mathbb C^\ast\), hence each is biholomorphic to \(\mathbb C^\ast\) by the projection \((z,w)\mapsto z\). Therefore each is irreducible.

\vspace{0.1cm}

For a graph
\(
w=e^{g(z)}
\)
inside \((\mathbb C^\ast)^2\), the logarithmic map is
\[
\Log(z,w)=\bigl(\log|z|,\log|w|\bigr)
=
\bigl(\log|z|,\operatorname{Re}g(z)\bigr).
\]
Thus the logarithmic limit set is obtained by studying all possible limiting directions of
\(
\bigl(\log|z|,\operatorname{Re}g(z)\bigr)
\)
as this vector becomes unbounded.
We begin with
\(
V_1=\{w=e^z\}.
\)
Here
\(
g(z)=z,
\)
so
\[
\Log(V_1)
=
\{(\log|z|,\operatorname{Re}z):z\in\mathbb C^\ast\}.
\]
Write \(z=re^{i\theta}\). Then
\(
\Log(z,e^z)=(\log r,r\cos\theta).
\)
If \(r\to0\), then \(\log r\to-\infty\) and \(r\cos\theta\to0\), so the limiting direction is
\(
(-1,0).
\)
If \(r\to+\infty\) and
\(
\dfrac{r\cos\theta}{\log r}\to\lambda\in\mathbb R,
\)
then
\[
\frac{(\log r,r\cos\theta)}
{\sqrt{(\log r)^2+(r\cos\theta)^2}}
\longrightarrow
\frac{(1,\lambda)}{\sqrt{1+\lambda^2}}.
\]
If
\(
\dfrac{r\cos\theta}{\log r}\to+\infty,
\)
the limiting direction is \((0,1)\), and if
\(
\dfrac{r\cos\theta}{\log r}\to-\infty,
\)
the limiting direction is \((0,-1)\). Conversely these are the only possibilities. Hence
\[
\mathscr L^\infty(V_1)
=
\{(-1,0)\}
\cup
\{
\frac{(1,\lambda)}{\sqrt{1+\lambda^2}}:\lambda\in\mathbb R
\}
\cup
\{(0,1),(0,-1)\}.
\]
Equivalently, \(\mathscr L^\infty(V_1)\) is the union of the point \((-1,0)\) and the closed right semicircle of \(S^1\).
We now consider
\(
V_2=\{w=e^{1/z}\}.
\)
Here
\(
g(z)=\frac1z,
\)
so
\[
\Log(V_2)
=
\{
(\log|z|,\operatorname{Re}\frac1z ):z\in\mathbb C^\ast
\}.
\]
Put
\(
s=\frac1z.
\)
Then
\(
\log|z|=-\log|s|.
\)
As \(z\to\infty\), one has \(s\to0\), so
\(
\operatorname{Re}s\to0
\)
and
\(
\log|z|\to+\infty.
\)
Therefore the end \(z\to\infty\) contributes the isolated direction
\(
(1,0).
\)
As \(z\to0\), equivalently \(s\to\infty\), we have
\[
\Log(z,e^{1/z})=(-\log|s|,\operatorname{Re}s).
\]
Writing \(s=Re^{i\theta}\), this becomes
\(
(-\log R,R\cos\theta).
\)
If
\(
\dfrac{R\cos\theta}{\log R}\to\lambda,
\)
then the limiting direction is
\(
\dfrac{(-1,\lambda)}{\sqrt{1+\lambda^2}}.
\)
If
\(
\dfrac{R\cos\theta}{\log R}\to+\infty,
\)
the limiting direction is \((0,1)\), and if
\(
\dfrac{R\cos\theta}{\log R}\to-\infty,
\)
the limiting direction is \((0,-1)\). Therefore
\[
\mathscr L^\infty(V_2)
=
\{(1,0)\}
\cup
\{
\frac{(-1,\lambda)}{\sqrt{1+\lambda^2}}:\lambda\in\mathbb R
\}
\cup
\{(0,1),(0,-1)\}.
\]
Equivalently, \(\mathscr L^\infty(V_2)\) is the union of the point \((1,0)\) and the closed left semicircle of \(S^1\).
We finally consider
\(
V_3=\{w=\exp (\sin\frac1z )\}.
\)
Here
\(
g(z)=\sin\frac1z.
\)
Again put
\(
s=\frac1z.
\)
Then
\(
\Log(V_3)
=
\left\{
\left(-\log|s|,\operatorname{Re}(\sin s)\right):s\in\mathbb C^\ast
\right\}.
\)

Writing \(s=a+ib\), one has
\(
\sin s=\sin a\cosh b+i\cos a\sinh b,
\)
and therefore
\(
\operatorname{Re}(\sin s)=\sin a\cosh b.
\)
The end \(z\to\infty\), equivalently \(s\to0\), contributes
\(
(1,0).
\)
The end \(z\to0\), equivalently \(s\to\infty\), contributes the closed left semicircle. Indeed, for any \(\lambda\in\mathbb R\), choose \(s_R=a_R+ib_R\) with \(a_R\sim R\), \(\sin a_R=\operatorname{sgn}(\lambda)\), and
\(
\cosh b_R=|\lambda|\log R
\)
when \(\lambda\neq0\). Then
\(
|s_R|\sim R
\)
and
\(
\operatorname{Re}(\sin s_R)=\lambda\log R.
\)
Thus the normalized logarithmic vector tends to
\(
\dfrac{(-1,\lambda)}{\sqrt{1+\lambda^2}}.
\)
Letting \(\lambda\to\pm\infty\) gives \((0,\pm1)\). Conversely, every sequence \(s_j\to\infty\) gives either a finite ratio
\(
\dfrac{\operatorname{Re}(\sin s_j)}{\log|s_j|}
\)
and hence a point of this left semicircle, or an infinite ratio and hence \((0,1)\) or \((0,-1)\). Therefore
\[
\mathscr L^\infty(V_3)
=
\{(1,0)\}
\cup
\{
\frac{(-1,\lambda)}{\sqrt{1+\lambda^2}}:\lambda\in\mathbb R
\}
\cup
\{(0,1),(0,-1)\}.
\]
Thus \(V_2\) and \(V_3\) have the same logarithmic limit set, even though their analytic behavior near \(z=0\) is different.

\begin{proposition}
The logarithmic limit sets are
\[
\mathscr L^\infty(\{w=e^z\})
=
\{(-1,0)\}
\cup
\{
\frac{(1,\lambda)}{\sqrt{1+\lambda^2}}:\lambda\in\mathbb R
\}
\cup
\{(0,1),(0,-1)\},
\]
\[
\mathscr L^\infty(\{w=e^{1/z}\})
=
\{(1,0)\}
\cup
\{
\frac{(-1,\lambda)}{\sqrt{1+\lambda^2}}:\lambda\in\mathbb R
\}
\cup
\{(0,1),(0,-1)\},
\]
and
\[
\mathscr L^\infty (\{w=\exp (\sin\frac1z )\})
=
\{(1,0)\}
\cup
\{
\frac{(-1,\lambda)}{\sqrt{1+\lambda^2}}:\lambda\in\mathbb R
\}
\cup
\{(0,1),(0,-1)\}.
\]
Hence \(w=e^z\) gives the closed right semicircle together with the opposite isolated direction, while \(w=e^{1/z}\) and \(w=\exp(\sin(1/z))\) give the closed left semicircle together with the opposite isolated direction.
\end{proposition}

\begin{proof}
The proof is exactly the computation above. For \(w=e^z\), the unbounded term \(\operatorname{Re}z\) occurs at the end \(z\to\infty\), while the end \(z\to0\) contributes only \((-1,0)\). For \(w=e^{1/z}\), the unbounded term \(\operatorname{Re}(1/z)\) occurs at \(z\to0\), while the end \(z\to\infty\) contributes only \((1,0)\). For \(w=\exp(\sin(1/z))\), the term \(\operatorname{Re}(\sin(1/z))\) also becomes unbounded at \(z\to0\), and its oscillation is large enough to realize the whole closed left semicircle. No other limiting directions are possible, because if \(z\) stays in a compact subset of \(\mathbb C^\ast\), both logarithmic coordinates remain bounded.
\end{proof}

The comparison reveals two important points. 
The hypersurface \(w=e^{1/z}\) has an essential singularity at \(z=0\), but its logarithmic limit set near \(z=0\) is the closed left semicircle, not merely a finite set of directions. Therefore the essential singularity already produces a continuum of logarithmic directions.
The hypersurface \(w=\exp(\sin(1/z))\) also has an essential singularity at \(z=0\). It has infinitely many local sheets over generic \(w\)-values inside the single topological end \(z\to0\). These infinitely many sheets are logarithmically parallel along suitable transverse levels, but the full logarithmic limit set again contains a continuum of directions. Thus infinite sheet behavior is accompanied by a large logarithmic limit set.
The hypersurface \(w=e^z\) is different in location but not in principle. Its exponential growth occurs at \(z\to\infty\), and this creates the closed right semicircle. The other end \(z\to0\) contributes only the isolated direction \((-1,0)\).

%%%%%%%%%%%%%%%%%%%%%%%%%%%%%%%%%%%%%%%%%%%%%%%%%%%%%%%%%%%%%%%%%%%%%%%%
%%%%%%%%%%%%%%%%%%%%%%%%%%%%%%%%%%%%%%%%%%%%%%%%%%%%%%%%%%%%%%%%%%%%%%%%

\end{document}